\documentclass[12pt,draftcls,onecolumn]{IEEEtran}

\input epsf

\usepackage{amsmath,amssymb,pstricks,color}
\usepackage{psfrag,graphicx,epsfig}
\usepackage{algorithm,algorithmic,xspace}
\usepackage[normal]{subfigure}

\newcommand{\nnum}{\nonumber}

\newcommand{\EQ}{\begin{eqnarray}}
\newcommand{\EN}{\end{eqnarray}}
\newcommand{\EQQ}{\begin{eqnarray*}}
\newcommand{\ENN}{\end{eqnarray*}}

\newcommand{\bremark}{\begin{remark} \begin{rm} }
\newcommand{\eremark}{ \end{rm} \rule{1mm}{2mm}
\end{remark} }
\newcommand{\btheorem}{\begin{theorem} \begin{rm} }
\newcommand{\etheorem}{ \end{rm} \rule{1mm}{2mm}
\end{theorem} }
\newcommand{\blemma}{\begin{lemma} \begin{rm} }
\newcommand{\elemma}{ \end{rm} \rule{1mm}{2mm}
\end{lemma} }
\newcommand{\bcorollary}{\begin{corollary} \begin{rm} }
\newcommand{\ecorollary}{ \end{rm} \rule{1mm}{2mm}
\end{corollary} }
\newcommand{\bdefinition}{\begin{definition}\begin{rm} }
\newcommand{\edefinition}{ \end{rm} \rule{1mm}{2mm}
\end{definition} }
\newcommand{\bproposition}{\begin{proposition} \begin{rm} }
\newcommand{\eproposition}{ \end{rm} \rule{1mm}{2mm}
\end{proposition} }
\newcommand{\bexample}{\begin{example} \begin{rm} }
\newcommand{\eexample}{ \end{rm} \rule{1mm}{2mm}
\end{example} }
\newcommand{\basm}{\begin{assumption} \begin{rm}}
\newcommand{\easm}{\end{rm} 
\end{assumption}}

\newcommand{\real}{\mathbb{R}}

\renewcommand{\natural}{\mathbb{N}}

\newcommand{\DD}{\mathcal{D}}

\newcommand{\GG}{\mathcal{G}}

\newcommand{\LL}{\mathcal{L}}

\newcommand{\dist}{\operatorname{dist}}

\newcommand{\until}[1]{\{1,\dots, #1\}}
\newcommand{\subscr}[2]{#1_{\textup{#2}}}

\newtheorem{theorem}{\bf Theorem}[section]
\newtheorem{lemma}{\bf Lemma}[section]
\newtheorem{definition}{\bf Definition}[section]
\newtheorem{remark}{\bf Remark}[section]
\newtheorem{corollary}{\bf Corollary}[section]
\newtheorem{proposition}{\bf Proposition}[section]
\newtheorem{example}{\bf Example}[section]
\newtheorem{assumption}{\bf Assumption}[section]

\newcommand\oprocendsymbol{\hbox{$\bullet$}}
\newcommand\oprocend{\relax\ifmmode\else\unskip\hfill\fi\oprocendsymbol}

\date{}

\begin{document}


\title{An approximate dual subgradient algorithm for multi-agent non-convex optimization}

\author{ Minghui Zhu and Sonia Mart{\'\i}nez \thanks{M. Zhu is with
    Laboratory for Information and Decision Systems,
    Massachusetts Institute of Technology, 77 Massachusetts Avenue, Cambridge MA, 02139, {\tt\small (mhzhu@mit.edu)}. S. Mart{\'\i}nez is with Department of
    Mechanical and Aerospace Engineering, University of California,
    San Diego, 9500 Gilman Dr, La Jolla CA, 92093, {\tt\small
      (soniamd@ucsd.edu)}.}}

\maketitle

\begin{abstract}
  We consider a multi-agent optimization problem where agents subject
  to local, intermittent interactions aim to minimize a sum of local objective
  functions subject to a global inequality constraint and a global
  state constraint set. In contrast to previous work, we do not
  require that the objective, constraint functions, and state
  constraint sets to be convex. In order to deal with time-varying
  network topologies satisfying a standard connectivity assumption, we
  resort to consensus algorithm techniques and the Lagrangian duality method.
  We slightly relax the requirement of exact consensus, and propose a
  distributed approximate dual subgradient algorithm to enable agents
  to asymptotically converge to a pair of primal-dual solutions to an
  approximate problem. To guarantee convergence, we assume that the
  Slater's condition is satisfied and the optimal solution set of the
  dual limit is singleton. We implement our algorithm over a source
  localization problem and compare the performance with existing
  algorithms.
\end{abstract}

\section{Introduction}\label{sec:introduction}

Recent advances in computation, communication, sensing and actuation
have stimulated an intensive research in networked multi-agent
systems. In the systems and controls community, this has translated
into how to solve global control problems, expressed by global
objective functions, by means of local agent actions. Problems
considered include multi-agent consensus or
agreement~\cite{AJ-JL-ASM:02,ROS-RMM:03c}, coverage
control~\cite{FB-JC-SM:09,JC-SM-TK-FB:02j}, formation
control~\cite{JAF-RMM:04,WR-RWB:08} and sensor
fusion~\cite{LX-SB-SL:05}.

The seminal work~\cite{DPB-JNT:97} provides a framework to tackle
optimizing a global objective function among different processors
where each processor knows the global objective function. In
multi-agent environments, a problem of focus is to minimize a sum of
local objective functions by a group of agents, where each function
depends on a common global decision vector and is only known to a
specific agent. This problem is motivated by others in distributed
estimation~\cite{RDN:03}~\cite{SSR-AN-VVV:08b}, distributed source
localization~\cite{MGR-RDN:04}, and network utility
maximization~\cite{KPK-AM-DT:98}. More recently, consensus techniques
have been proposed to address the issues of switching topologies,
asynchronous computation and coupling in objective functions; see for
instance~\cite{AN-AO:09,AN-AO-PAP:08,MZ-SM:09c}.
More specifically, the paper~\cite{AN-AO:09} presents the first
analysis of an algorithm that combines average consensus schemes with
subgradient methods. Using projection in the algorithm
of~\cite{AN-AO:09}, the authors in~\cite{AN-AO-PAP:08} further address
a more general scenario that takes local state constraint sets into
account. Further, in~\cite{MZ-SM:09c} we develop two distributed
primal-dual subgradient algorithms, which are based on saddle-point
theorems, to analyze a more general situation that incorporates global
inequality and equality constraints. The aforementioned algorithms are
extensions of classic (primal or primal-dual) subgradient methods
which generalize gradient-based methods to minimize non-smooth
functions. This requires the optimization problems under consideration
to be convex in order to determine a global optimum.

The focus of the current paper is to relax the convexity assumption
in~\cite{MZ-SM:09c}. In order to deal with all aspects of our
multi-agent setting, our method integrates Lagrangian dualization,
subgradient schemes, and average consensus algorithms. Distributed
function computation by a group of anonymous agents interacting
intermittently can be done via agreement
algorithms~\cite{FB-JC-SM:09}. However, agreement algorithms are
essentially convex, and so we are led to the investigation of
nonconvex optimization solutions via dualization. The techniques of
dualization and subgradient schemes have been popular and efficient
approaches to solve both convex programs (e.g., in~\cite{DPB:09}) and
nonconvex programs (e.g., in~\cite{RSB:11,RSB-CYK:07}).

\emph{Statement of Contributions.} Here, we investigate a
multi-agent optimization problem where agents desire to agree upon a
global decision vector minimizing the sum of local objective functions
in the presence of a global inequality constraint and a global state
constraint set. Agent interactions are changing with time. The objective, constraint
functions, as well as the state-constraint set, can be nonconvex. To
deal with both nonconvexity and time-varying interactions, we first
define an approximated problem where the exact consensus is slightly
relaxed. We then propose a distributed dual subgradient algorithm to
solve it, where the update rule for local dual estimates combines a
dual subgradient scheme with average consensus algorithms, and local
primal estimates are generated from local dual optimal solution
sets. This algorithm is shown to asymptotically converge to a pair of
primal-dual solutions to the approximate problem under the following
assumptions: firstly, the Slater's condition is satisfied; secondly,
the optimal solution set of the dual limit is singleton; thirdly,
dynamically changing network topologies satisfy some standard
connectivity condition.

A conference version of this manuscript was published
in~\cite{MZ-SM:10b}. Main differences are the following: (i) by
assuming that the optimal solution set of the dual limit is a
singleton, and changing the update rule in the dual estimates, we are
able to determine a global solution in contrast to an approximate
solution in~\cite{MZ-SM:10b}; (ii) we present a simple criterion to
check the new sufficient condition for nonconvex quadratic
programming; (iii) new simulation results of our algorithm on a source
localization example and a comparison of its performance with existing
algorithms are performed. 

\section{Problem formulation and preliminaries}\label{sec:formulation}

Consider a networked multi-agent system where agents are labeled by
$i\in V := \until N$. The multi-agent system operates in a synchronous way
at time instants $k\in \natural \cup \{0\}$, and its topology will be
represented by a directed weighted graph ${\GG}(k) = (V,E(k),A(k))$,
for $k \ge 0$. Here, $A(k) := [a^i_j(k)] \in \real^{N\times N}$ is the
adjacency matrix, where the scalar $a^i_j(k)\geq0$ is the weight assigned to the
edge $(j,i)$ pointing from agent~$j$ to agent~$i$, and $E(k)\subseteq V\times V\setminus {\rm diag}(V)$ is the set of edges with non-zero weights. The set of in-neighbors of
agent $i$ at time $k$ is denoted by ${\mathcal{N}}_i(k) = \{j\in V \;
| \; (j,i)\in E(k) \text{ and } j\neq i\}$. Similarly, we define the
set of out-neighbors of agent $i$ at time $k$ as ${\mathcal{N}}_i^{\rm
  out}(k) = \{j\in V \; | \; (i,j)\in E(k) \text{ and } j\neq i\}$. We
here make the following assumptions on network communication
graphs:

\begin{assumption} [Non-degeneracy]
  There exists a constant $\alpha>0$ such that
  $a_i^i(k)\geq\alpha$, and $a^i_j(k)$, for $i\neq j$, satisfies
  $a^i_j(k)\in \{0\}\cup [\alpha,\;1],\;$ for all
  $k\geq0$. \label{asm20}
\end{assumption}

\begin{assumption}[Balanced Communication]
It holds that $\sum_{j\in V}
  a_j^i(k)=1$ for all $i\in V$ and $k\geq0$, and $\sum_{i\in V}
  a_j^i(k)=1$ for all $j\in V$ and $k\geq0$.\label{asm30}
\end{assumption}

\begin{assumption} [Periodical Strong Connectivity]
  There is a positive integer $B$ such that, for all $k_0\geq0$,
  the directed graph $(V,\bigcup_{k=0}^{B-1}E(k_0 + k))$ is strongly
  connected. \label{asm10}
\end{assumption}

The above network model is standard to characterize a networked
multi-agent system, and has been widely used in the analysis of
average consensus algorithms; e.g., see~\cite{ROS-RMM:03c,AO-JNT:07},
and distributed optimization
in~\cite{AN-AO-PAP:08,MZ-SM:09c}. Recently, an algorithm is given
in~\cite{BG-JC:09} which allows agents to construct a balanced graph
out of a non-balanced one under certain assumptions.

The objective of the agents is to cooperatively solve the
following primal problem ($P$):
\begin{align}
  \min_{z\in {\real}^n } \sum_{i\in V} f_i(z),\quad
  {\rm s.t.} \quad g(z)\leq 0, \quad z\in X,\label{e3}
\end{align}
where $z\in\real^n$ is the global decision vector. The function $f_i :
{\real}^n \rightarrow {\real}$ is only known to agent~$i$, continuous,
and referred to as the objective function of agent~$i$. The set
$X\subseteq {\real}^n$, the state constraint set, is compact. The
function $g : {\real}^n \rightarrow {\real}^m$ are continuous, and the
inequality $g(z) \leq 0$ is understood component-wise; i.e.,
$g_{\ell}(z)\leq0$, for all $\ell \in \until m$, and represents a
global inequality constraint. We will denote $f(z) := \sum_{i\in V}
f_i(z)$ and $Y := \{z\in{\real}^n \; | \; g(z)\leq0\}$.  We will
assume that the set of feasible points is non-empty; i.e., $X \cap Y
\neq \emptyset$. Since $X$ is compact and $Y$ is closed, then we can
deduce that $X\cap Y$ is compact. The continuity of $f$ follows from
that of $f_i$. In this way, the optimal value $p^*$ of the problem
($P$) is finite and $X^*$, the set of primal optimal points, is
non-empty. We will also assume the following Slater's condition holds:

\begin{assumption}[Slater's Condition] There exists a vector
  $\bar{z}\in X$ such that $g(\bar{z}) < 0$. Such
  $\bar{z}$ is referred to as a Slater vector of the problem ($P$).\label{asm5}
\end{assumption}

\begin{remark} All the agents can agree upon a common Slater vector
  $\bar{z}$ through a maximum-consensus scheme. This can be easily
  implemented as part of an initialization step, and thus the
  assumption that the Slater vector is known to all agents does not
  limit the applicability of our algorithm. Specifically, the maximum-consensus algorithm is
  described as follows:

  Initially, each agent~$i$ chooses a Slater vector $z_i(0)\in X$ such
  that $g(z_i(0)) < 0$.  At every time $k\geq0$, each agent~$i$ updates
  its estimates by using the rule of $z_i(k+1) = \max_{j\in{\mathcal{N}}_i(k)\cup\{i\}} z_j(k)$,
where we use the following relation for vectors: for $a,b\in
\real^n$, $a < b$ if and only if there is some $\ell \in
\until{n-1}$ such that $a_{\kappa}=b_{\kappa}$ for all $\kappa <
\ell$ and $a_{\ell} < b_{\ell}$.

The periodical strong connectivity assumption~\ref{asm10} ensures that
after at most $(N-1)B$ steps, all the agents reach the consensus;
i.e., $z_i(k) = \max_{j\in V} z_j(0)$ for all $k\geq(N-1)B$. In the
remainder of this paper, we assume that the Slater vector $\bar{z}$ is
known to all the agents. \oprocend\label{rem2}
\end{remark}

In~\cite{MZ-SM:09c}, in order to solve the convex case of the problem
($P$) (i.e.; $f_i$ and $g$ are convex functions and $X$ is a convex
set), we propose two distributed primal-dual subgradient algorithms
where primal (resp. dual) estimates move along subgradients
(resp. supergradients) and are projected onto convex sets. The absence
of convexity impedes the use of the algorithms in~\cite{MZ-SM:09c}
since, on the one hand, (primal) gradient-based algorithms are easily
trapped in local minima; on the other hand, projection maps may not be
well-defined when (primal) state constraint sets are nonconvex. In the
sequel, we will employ Lagrangian dualization, subgradient methods and
average consensus schemes to design a distributed algorithm which is
able to find an approximate solution to the problem ($P$).

Towards this end, we construct a directed cyclic graph $\subscr{\GG}{cyc} := (V,
\subscr{E}{cyc})$ where $|\subscr{E}{cyc}| = N$. We assume that each
agent has a unique in-neighbor (and out-neighbor). The out-neighbor
(resp. in-neighbor) of agent $i$ is denoted by $i_D$
(resp. $i_U$). With the graph $\subscr{\GG}{cyc}$, we will study the
following approximate problem of problem ($P$):
\begin{align}
  & \min_{(x_i) \in \real^{nN}} \sum_{i\in V} f_i(x_i),\nnum\\
  & {\rm s.t.} \quad g(x_i)\leq 0, \quad - x_i + x_{i_D} - \Delta \leq 0,\quad
  x_i - x_{i_D} - \Delta \leq 0, \quad x_i\in X,\quad \forall i\in V,\label{e6}
\end{align} where $\Delta := \delta\textbf{1}$, with $\delta$ a small
positive scalar, and $\textbf{1}$ is the column vector of $n$
ones. The problem~\eqref{e6} provides an approximation of the problem
($P$), and will be referred to as problem ($P_{\Delta}$). In
particular, the approximate problem~\eqref{e6} reduces to the problem
($P$) when $\delta = 0$. Its optimal value and the set of optimal
solutions will be denoted by $p^*_{\Delta}$ and $X^*_{\Delta}$,
respectively. Similarly to the problem ($P$), $p^*_{\Delta}$ is finite
and $X^*_{\Delta}\neq\emptyset$.

\begin{remark}
  The cyclic graph $\subscr{\GG}{cyc}$ can be replaced by any strongly
  connected graph $\GG$. Given $\GG$, each agent $i$ is endowed with
  two inequality constraints: $x_i - x_j - \Delta \leq 0$ and $- x_i +
  x_j - \Delta \leq 0$, for each out-neighbor $j$. This set of
  inequalities implies that any feasible solution $x = (x_i)_{i\in V}$
  of problem ($P_{\Delta}$) satisfies the approximate consensus; i.e.,
  $\max_{i,j\in V}\|x_i - x_j\| \leq N\delta$. For simplicity, we will
  use the cyclic graph $\subscr{\GG}{cyc}$, with a minimum number
  of constraints, as the initial graph. \oprocend\label{rem1}
\end{remark}

\subsection{Dual problems}

Before introducing dual problems, let us denote by $\Xi' :=
\real^m_{\geq0}\times\real^{nN}_{\geq0}\times\real^{nN}_{\geq0}$, $\Xi
:=
\real^{mN}_{\geq0}\times\real^{nN}_{\geq0}\times\real^{nN}_{\geq0}$,
$\xi_i:=(\mu_i,\lambda,w)\in \Xi'$, $\xi := (\mu,\lambda,w)\in\Xi$
and $x := (x_i)\in X^N$. The dual problem ($D_{\Delta}$) associated
with $(P_\Delta)$ is given by
\begin{align} \max_{\mu,\lambda,w} Q(\mu,\lambda,w),\quad
  {\rm s.t.}\quad \mu,\lambda,w\geq0,\label{e14}
\end{align}
where $\mu := (\mu_i)\in\real^{mN}$, $\lambda :=
(\lambda_i)\in\real^{nN}$ and $w := (w_i)\in \real^{nN}$. Here, the
dual function $Q : \Xi \rightarrow\real$ is given as
$Q(\xi)\equiv Q(\mu,\lambda,w) := \inf_{x\in X^N} {\LL}(x,\mu,\lambda,w)$,
where ${\LL} : \real^{nN}\times\Xi\rightarrow\real$ is the Lagrangian
function
\begin{align*}
  \LL(x,\xi)\equiv{\LL}(x,\mu,\lambda,w) := \sum_{i\in V} \big(f_i(x_i) + \langle\mu_i, g(x_i)\rangle
  + \langle \lambda_i, - x_i + x_{i_D} - \Delta\rangle + \langle w_i,
  x_i - x_{i_D} -\Delta\rangle\big).
\end{align*}
We denote the dual optimal value of the problem ($D_{\Delta}$) by
$d^*_{\Delta}$ and the set of dual optimal solutions by
$D^*_{\Delta}$. We endow each agent $i$ with the local Lagrangian function
$\LL_i : \real^n\times\Xi'\rightarrow\real$ and the local dual
function $Q_i : \Xi'\rightarrow\real$ defined by
\begin{align*}
  \LL_i(x_i,\xi_i) &:= f_i(x_i) + \langle\mu_i, g(x_i)\rangle
  + \langle -\lambda_i+\lambda_{i_U}, x_i\rangle
  + \langle w_i-w_{i_U}, x_i\rangle - \langle \lambda_i, \Delta \rangle - \langle w_i, \Delta \rangle,\nnum\\
  Q_i(\xi_i) &:= \inf_{x_i\in X}\LL_i(x_i,\xi_i).
\end{align*}

In the approximate problem ($P_{\Delta}$), the introduction of $-\Delta \leq
x_i - x_{i_D} \leq \Delta$, $i\in V$, renders the $f_i$ and $g$
separable. As a result, the global dual function $Q$ can be decomposed
into a simple sum of the local dual functions $Q_i$. More precisely,
the following holds:
\begin{align*}
  Q(\xi) = \inf_{x\in X^N}\sum_{i\in V} \big(f_i(x_i) + \langle\mu_i, g(x_i)\rangle
  + \langle \lambda_i, - x_i + x_{i_D} - \Delta\rangle + \langle w_i, x_i - x_{i_D} -\Delta\rangle\big).
\end{align*}

Notice that in the sum of $\sum_{i\in V}\langle \lambda_i, - x_i + x_{i_D} - \Delta\rangle$, each $x_i$ for any $i\in V$ appears in two terms: one is $\langle \lambda_i, - x_i + x_{i_D} - \Delta\rangle$, and the other is $\langle \lambda_{i_U}, - x_{i_U} + x_i - \Delta\rangle$. With this observation, we regroup the terms in the summation in terms of $x_i$, and have the following: \begin{align}
  &Q(\xi) =
  \inf_{x\in X^N}\sum_{i\in V} \big(f_i(x_i) + \langle\mu_i, g(x_i)\rangle
  + \langle -\lambda_i+\lambda_{i_U}, x_i\rangle + \langle w_i-w_{i_U}, x_i\rangle - \langle \lambda_i, \Delta \rangle - \langle w_i, \Delta \rangle\big)\nnum\\
  & = \sum_{i\in V} \inf_{x_i\in X} \big(f_i(x_i) + \langle\mu_i, g(x_i)\rangle
  + \langle -\lambda_i+\lambda_{i_U}, x_i\rangle + \langle w_i-w_{i_U}, x_i\rangle - \langle \lambda_i, \Delta \rangle - \langle w_i, \Delta \rangle\big)\nnum\\
  &= \sum_{i\in V}Q_i(\xi_i).\label{e11}
\end{align}

Note that $\sum_{i\in V}Q_i(\xi_i)$ is not separable since $Q_i$
depends on neighbor's multipliers $\lambda_{i_U}$, $w_{i_U}$.

\subsection{Dual solution sets}

The Slater's condition ensures the boundedness of dual solution sets for
convex optimization; e.g.,~\cite{JBHU-CL:96,AN-AO:08b}. We will shortly see that the
Slater's condition plays the same role in nonconvex optimization. To
achieve this, we define the function $\hat{Q}_i :
\real^m_{\geq0}\times\real^n_{\geq0}\times\real^n_{\geq0} \rightarrow
\real$ as follows:
\begin{align*}
  \hat{Q}_i(\mu_i,\lambda_i,w_i) = \inf_{x_i\in X, x_{i_D}\in X}\big(f_i(x_i) + \langle\mu_i, g(x_i)\rangle
  + \langle \lambda_i, - x_i + x_{i_D} - \Delta\rangle + \langle w_i,
  x_i - x_{i_D} -\Delta\rangle\big).
\end{align*}

Let $\bar{z}$ be a Slater vector for problem ($P$). Then $\bar{x} =
(\bar{x}_i)\in X^N$ with $\bar{x}_i = \bar{z}$ is a Slater vector of
the problem ($P_{\Delta}$). Similarly to (3) and (4)
in~\cite{MZ-SM:09c}, which make use of Lemma~3.2 in the same paper, we
have that for any $\mu_i,\lambda_i,w_i\geq0$, it holds that
\begin{align} \max_{\xi\in D^*_{\Delta}}\|\xi\| \leq N\max_{i\in
    V}\frac{f_i(\bar{z})-\hat{Q}_i(\mu_i,\lambda_i,w_i)}{\beta(\bar{z})},\label{e31}
\end{align}
where $\beta(\bar{z}):=\min\{\min_{\ell\in\until{m}}-g_{\ell}(\bar{z}),\delta\}$.
Let $\mu_i$, $\lambda_i$ and $w_i$ be zero
in~\eqref{e31}, and it leads to the following upper bound on $D^*_{\Delta}$:
\begin{align}
  \max_{\xi\in D^*_{\Delta}}\|\xi\| \leq N\max_{i\in
    V}\frac{f_i(\bar{z})-\hat{Q}_i(0,0,0)}
  {\beta(\bar{z})},\label{e32}
\end{align}
where $\hat{Q}_i(0,0,0) = \inf_{x_i\in X}f_i(x_i)$ and it can be
computed locally. We denote
\begin{align}\gamma_i(\bar{z}) : =
  \frac{f_i(\bar{z})-\hat{Q}_i(0,0,0)}{\beta(\bar{z})}.\label{e12}
\end{align}

Since $f_i$ and $g$ are continuous and $X$ is compact, then that $Q_i$
is continuous; e.g., see Theorem 1.4.16
in~\cite{JPA-HF:90}. Similarly, $Q$ is continuous. Since
$D^*_{\Delta}$ is also bounded, then we have that
$D^*_{\Delta}\neq\emptyset$.

\begin{remark} The requirement of exact agreement on $z$ in the
  problem $P$ is slightly relaxed in the problem $P_{\Delta}$ by
  introducing a small positive scalar $\delta$. In this way, the
  global dual function $Q$ is a sum of the local dual functions $Q_i$,
  as in~\eqref{e11}; $D^*_{\Delta}$ is non-empty and uniformly
  bounded. These two properties play important roles in the devise of
  our subsequent algorithm.\oprocend\label{rem4}
\end{remark}

\subsection{Other notation}

Define the set-valued map $\Omega_i : \Xi'\rightarrow 2^{X}$ as
$\Omega_i(\xi_i) := {\rm argmin}_{x_i\in X}\LL_i(x_i,\xi_i)$; i.e.,
given $\xi_i$, the set $\Omega_i(\xi_i)$ is the collection of
solutions to the following local optimization problem:
\begin{align}
  \min_{x_i\in X}\LL_i(x_i,\xi_i).\label{e39}
\end{align}
Here, $\Omega_i$ is referred to as the \emph{marginal map} of agent
$i$. Since $X$ is compact and $f_i$, $g$ are continuous, then
$\Omega_i(\xi_i)\neq \emptyset$ in~\eqref{e39} for any
$\xi_i\in\Xi'$. In the algorithm we will develop in next section, each
agent is required to obtain \emph{one} (globally) optimal solution and
the optimal value the local optimization problem~\eqref{e39} at each
iterate. We assume that this can be easily solved, and this is the
case for problems of $n=1$, or $f_i$ and $g$ being smooth (the
extremum candidates are the critical points of the objective function
and isolated corners of the boundaries of the constraint regions) or
having some specific structure which allows the use of global
optimization methods such as branch and bound algorithms.

In the space $\real^n$, we define the distance between a point
$z\in\real^n$ to a set $A\subset\real^n$ as $\dist(z,A) := \inf_{y\in
  A}\|z-y\|$, and the Hausdorff distance between two sets
$A,B\subset\real^n$ as $\dist(A,B) := \max\{\sup_{z\in A}\dist(z,B),
\sup_{y\in B}\dist(A,y)\}$. We denote by $B_{\mathcal{U}}(A,r) :=
\{u\in {\mathcal{U}} \; | \; \dist(u,A) \leq r\}$ and
$B_{2^{\mathcal{U}}}(A,r) := \{U\in 2^{\mathcal{U}} \; | \; \dist(U,A)
\leq r\}$ where ${\mathcal{U}}\subset\real^n$.

\section{Distributed approximate dual subgradient
  algorithm}\label{sec:algorithm}

In this section, we devise a distributed approximate dual subgradient
algorithm which aims to find a pair of primal-dual solutions to the
approximate problem ($P_{\Delta}$). 

For each agent $i$, let $x_i(k)\in\real^n$ be the estimate of the
primal solution $x_i$ to the approximate problem ($P_\Delta$) at time
$k \ge 0$, $\mu_i(k)\in\real^m_{\geq0}$ be the estimate of the
multiplier on the inequality constraint $g(x_i) \leq 0$,
$\lambda^i(k)\in\real^{nN}_{\geq0}$ (resp.
$w^i(k)\in\real^{nN}_{\geq0}$)\footnote{We will use the superscript
  $i$ to indicate that $\lambda^i(k)$ and $w^i(k)$ are estimates of
  some global variables.} be the estimate of the multiplier associated
with the collection of the local inequality constraints $- x_j +
x_{j_D} - \Delta \leq 0$ (resp. $x_j - x_{j_D} - \Delta \leq 0$), for
all $j\in V$. We let $\xi_i(k) :=
(\mu_i(k)^T,\lambda^i(k)^T,w^i(k)^T)^T\in \Xi'$, for $i\in V$ to be
the collection of dual estimates of agent~$i$. And denote $v_i(k) :=
(\mu_i(k)^T,v^i_{\lambda}(k)^T,v^i_w(k)^T)^T\in \Xi'$ where
$v^i_{\lambda}(k) := \sum_{j\in V}a^i_j(k)
\lambda^j(k)\in\real^{nN}_{\geq0}$ and $v^i_w(k) := \sum_{j\in
  V}a^i_j(k) w^j(k)\in \real^{nN}_{\geq0}$ are convex combinations of
dual estimates of agent~$i$ and its neighbors at time~$k$.

At time $k$, we associate each agent~$i$ a supergradient vector
$\DD_i(k)$ defined as\\ $\DD_i(k) :=
(\DD^i_{\mu}(k)^T,\DD^i_{\lambda}(k)^T,\DD^i_w(k)^T)^T$, where
$\DD^i_{\mu}(k) := g(x_i(k))\in\real^m$, $\DD^i_\lambda(k)$ has
components $\DD^i_{\lambda}(k)_i := -\Delta - x_i(k)\in \real^{n}$,
$\DD^i_{\lambda}(k)_{i_U} := x_i(k)\in \real^n$, and
$\DD^i_{\lambda}(k)_j = 0\in \real^n$ for $j\in V\setminus\{i,i_U\}$,
while the components of $\DD^i_w(k)$ are given by: $\DD^i_w(k)_i := -
\Delta + x_i(k)\in \real^{n}$, $\DD^i_w(k)_{i_U} := -x_i(k)\in
\real^{n}$, and $\DD^i_w(k)_j = 0\in \real^{n}$, for $j\in
V\setminus\{i,i_U\}$. For each agent~$i$, we define the set $M_i :=
\{\xi_i\in\Xi' \; | \; \|\xi_i\|\leq \gamma + \theta_i\}$ for some
$\theta_i > 0$ where $\gamma := N\max_{i\in V}\gamma_i(\bar{z})$. Let
$P_{M_i}$ to be the projection onto the set $M_i$. It is easy to check
that $M_i$ is closed and convex, and thus the projection map $P_{M_i}$
is well-defined.

The \emph{Distributed Approximate Dual Subgradient (DADS)} Algorithm
is described in Table~\ref{ta:DADS}.

%
%
%

\begin{algorithm}[htbp]
  \caption{The Distributed Approximate Dual Subgradient Algorithm} \label{ta:DADS}

\begin{algorithmic}

  \REQUIRE Initially, all the agents agree upon some $\delta>0$ in the
  approximate problem ($P_{\Delta}$). Each agent $i$ chooses a common
  Slater vector $\bar{z}$, computes $\gamma_i(\bar{z})$ and obtains
  $ \gamma = N\max_{i\in V}\gamma_i(\bar{z})$ through a max-consensus
  algorithm where $\gamma_i(\bar{z})$ is given in~\eqref{e12}. After
  that, each agent $i$ chooses initial states $x_i(0)\in X$ and
  $\xi_i(0)\in \Xi'$.  \ENSURE At each time~$k$, each agent $i$
  executes the following steps:
\end{algorithmic}

\begin{algorithmic}[1]

  \STATE For each $k\geq1$, given $v_i(k)$, solve the local
  optimization problem~\eqref{e39}, obtain a solution
  $x_i(k)\in\Omega_i(v_i(k))$ and the dual optimal value
  $Q_i(v_i(k))$.

  \STATE For each $k\geq0$, generate the dual estimate $\xi_i(k+1)$
  according to the following rule: \begin{align} \xi_i(k+1) =
    P_{M_i}[v_i(k) + \alpha(k) \DD_i(k)],\label{e48}\end{align} where
  the scalar $\alpha(k)\geq0$ is a step-size.

\STATE Repeat for $k = k+1$.
\end{algorithmic}
\end{algorithm}

\begin{remark} The DADS algorithm is an extension of the classical
  dual algorithm, e.g., in~\cite{BTP:67} and~\cite{DPB:09} to the
  multi-agent setting and nonconvex case. In the initialization of
  the DADS algorithm, the value $\gamma$ serves as
  an upper bound on $D_{\Delta}^*$. In Step~$1$, \emph{one solution}
  in $\Omega_i(v_i(k))$ is needed, and it is unnecessary to compute
  the whole set $\Omega_i(v_i(k))$. \oprocend\label{rem3}
\end{remark}

In order to assure the primal convergence, we will assume that the
dual estimates converge to the set where each has a single optimal
solution.
\begin{definition}[Singleton optimal dual solution set] The set of
  $D_s^*\subseteq \real^{(m+2n)N}$ is the collection of $\xi$ such
  that the set $\Omega_i(\xi_i)$ is a singleton, where $\xi_i =
  (\mu_i,\lambda,w)$ for each $i\in V$. \oprocend \label{def1}
\end{definition}


%

The primal and dual estimates in the DADS algorithm will be shown to
asymptotically converge to a pair of primal-dual solutions to the
approximate problem ($P_{\Delta}$). We formally state this in the
following theorem:

\begin{theorem}[Convergence properties of the DADS algorithm] Consider
  the problem ($P$) and the corresponding approximate problem
  ($P_{\Delta}$) with some $\delta > 0$. We let the non-degeneracy
  assumption~\ref{asm20}, the balanced communication
  assumption~\ref{asm30} and the periodic strong connectivity
  assumption~\ref{asm10} hold. In addition, suppose the Slater's
  condition~\ref{asm5} holds for the problem ($P$). Consider the dual
  sequences of $\{\mu_i(k)\}$, $\{\lambda^i(k)\}$, $\{w^i(k)\}$ and
  the primal sequence of $\{x_i(k)\}$ of the distributed approximate
  dual subgradient algorithm with $\{\alpha(k)\}$ satisfying
  $\displaystyle{\lim_{k\rightarrow+\infty}\alpha(k) = 0}$,
  $\displaystyle{\sum_{k=0}^{+\infty}\alpha(k) = +\infty}$,
  $\displaystyle{\sum_{k=0}^{+\infty}\alpha(k)^2 < +\infty}$.
  \begin{enumerate}
  \item (Dual estimate convergence) There exists a dual solution
    $\xi^*\in D_{\Delta}^*$ where $\xi^* := (\mu^*,\lambda^*,w^*)$ and
    $\mu^*:=(\mu_i^*)$ such that the following holds for all $i\in V$:
  \begin{align*}{\lim_{k\rightarrow+\infty}\|\mu_i(k) - \mu_i^*\|
    = 0},\quad{\lim_{k\rightarrow+\infty}\|\lambda^i(k) -
    \lambda^*\| = 0},\quad{\lim_{k\rightarrow+\infty}\|w^i(k) - w^*\| =
    0}.\end{align*}
\item (Primal estimate convergence) If the dual solution satisfies
  $\xi^*\in D_s^*$, i.e.~$\Omega_i(\xi^*_i)$ is a singleton for all
  $i\in V$, then there is $x^*\in X_{\Delta}^*$ with $x^*:=(x_i^*)$
  such that, for all $i\in V$:
  \begin{align*}{\lim_{k\rightarrow+\infty}\|x_i(k) - x_i^*\| =
    0}.\end{align*}
    \end{enumerate}\label{the2}
\end{theorem}

\section{Discussion}

Before proceeding with the technical proofs for Theorem~\ref{the2}, we
would like to make the following observations. First, our methodology
is motivated by the need of solving a nonconvex problem in a
distributed way by a group of agents whose interactions change with
time. This places a number of restrictions on the type of solutions
that one can find. Time-varying interactions of anonymous agents can
be currently solved via agreement algorithms; however these are
inherently convex operations, which does not work well in nonconvex
settings. To overcome this, one can resort to dualization. Admittedly,
zero duality gap does not hold in general for nonconvex problems. A
possibility would be to resort to nonlinear augmented Lagrangians, for
which strong duality holds in a broad class of
programs~\cite{RSB:11,RSB-CYK:07,RTR-RJBW:98}. However, we find here
another problem, as a distributed solution using agreement requires
separability, as the one ensured by the linear Lagrangians we use
here. Thus, we have looked for alternative assumptions that can be
easier to check and allow the dualization approach to work.

More precisely, Theorem~\ref{the2} shows that dual estimates always
converge to a dual optimal solution. The convergence of primal
estimates requires an additional assumption that the dual limit has a
single optimal solution. Let us refer to this assumption as \emph{the
  singleton dual optimal solution set} (SD for short). This assumption
may not be easy to check \emph{a priori}, however it is of similar
nature as existing algorithms for nonconvex
optimization. In~\cite{RSB:11} and~\cite{RSB-CYK:07}, subgradient methods are defined in terms of (nonlinear) augmented Lagrangians, and it is shown that every accumulation point of the
primal sequence is a primal solution provided that the dual function
is required to be differentiable at the dual limit. An open question
is how to resolve the above issues imposed by the multi-agent setting
with less stringent conditions on the nature of the nonconvex
optimization problem.

In the following, we study a class of nonconvex quadratic programs
for which a sufficient condition guarantees that the SD assumption
holds. Nonconvex quadratic programs hold great importance from both
theoretic and practical aspects. In general, nonconvex quadratic
programs are NP-hard, and please refer to~\cite{PMP-SAV:91} for detailed
discussion. The aforementioned sufficient condition only requires checking
the positive definiteness of a matrix.

Consider the following nonconvex quadratic program:
\begin{align}
  &\min_{z\in \cap_{i\in V}X_i} f(z) = \sum_{i\in V}f_i(z) = \sum_{i\in V}\big(\|z\|^2_{P_i} + 2\langle q_i, z\rangle\big),\nnum\\
  &{\rm s.t.}\quad \|z\|^2_{A_{i,\ell_i}} + 2\langle
  b_{i,\ell_i}, z\rangle + c_{i,\ell_i} \leq 0, \quad \ell_i =
  1,\cdots,m_i,\label{e1}
\end{align}
where $\|z\|^2_{A_{i,\ell_i}} \triangleq z^T A_{i,\ell_i} z$ and ${A_{i,\ell_i}}$ are real and
symmetric matrices. The approximate problem of $P_{\Delta}$ is given by \begin{align}
  &\min_{x\in\real^{2N}} \sum_{i\in V}f_i(x_i) = \sum_{i\in V}\big(\|x_i\|^2_{P_i} + 2\langle q_i, x_i\rangle\big),\nnum\\
  &{\rm s.t.}\quad \|x_i\|^2_{A_{i,\ell_i}} + 2\langle
  b_{i,\ell_i}, x_i\rangle + c_{i,\ell_i} \leq 0, \quad \ell_i =
  1,\cdots,m_i,\nnum\\
  &\quad\quad - x_i + x_{i_D} - \Delta \leq 0,\quad
  x_i - x_{i_D} - \Delta \leq 0,\quad x_i\in X_i,\quad i\in V.\label{e2}
\end{align}


We introduce the dual multipliers $(\mu,\lambda,w)$ as before. The local Lagrangian function $\LL_i$ can be written as follows:
\begin{align}
  \LL_i(x_i,\xi_i) \triangleq \|x_i\|^2_{P_i + \sum_{\ell_i=1}^{m_i}
    \mu_{i,\ell_i}A_{i,\ell_i}} + \langle \zeta_i,x_i\rangle,\nnum
\end{align} where the term independent of $x_i$ is dropped and
$\zeta_i$ is a linear function of $\xi_i = (\mu_i,\lambda,w)$. The
dual function and dual problem can be defined as before. Consider any
dual optimal solution $\xi^*$. If for all $i\in V$:

(P1) $P_i + \sum_{\ell_i=1}^{m_i}\mu_{i,\ell_i}^*A_{i,\ell_i}$ is positive
definite;

(P2) $x^*_i = \big(P_i + \sum_{\ell_i=1}^{m_i}
    \mu_{i,\ell_i}^*A_{i,\ell_i}\big)^{-1}\zeta_i^*\in X_i$;

\noindent then the SD assumption holds. The properties (P1) and (P2)
are easy to verify in a distributed way once a dual solution $\xi^*$
is obtained. We would like to remark that (P1) is used
in~\cite{DYG-NR-HS:09} to determine the unique global optimal solution
via canonical duality when $X$ is absent.

\section{Convergence analysis}\label{sec:analysis}

This section provides a guide to the complete analysis of
Theorem~\ref{the2}.  Recall that $g$ is continuous and $X$ is
compact. Then there are $G, H > 0$ such that $\|g(x)\| \leq G$ and
$\|x\| \leq H$ for all $x\in X$. We start our analysis from the
computation of supergradients of $Q_i$. Due to space reasons, we will
omit most technical proofs; these can be found in the enlarged
version~\cite{MZ-SM:10-arxiv}.

\begin{lemma}[Supergradient computation]
  If $\bar{x}_i\in \Omega_i(\bar{\xi}_i)$, then
  $\big(g(\bar{x}_i)^T, (-\Delta - \bar{x}_i)^T, \bar{x}_i^T,
  (\bar{x}_i-\Delta)^T, -\bar{x}_i^T)^T$ is a supergradient
  of $Q_i$ at $\bar{\xi}_i$; i.e., the following holds for any
  $\xi_i\in\Xi'$:
  \begin{align}
    Q_i(\xi_i) - Q_i(\bar{\xi}_i)&\leq \langle g(\bar{x}_i),\mu_i -
    \bar{\mu}_i\rangle + \langle -\Delta - \bar{x}_i, \lambda_i -
    \bar{\lambda}_i\rangle\nnum\\ &+ \langle \bar{x}_i, \lambda_{i_U}
    - \bar{\lambda}_{i_U}\rangle + \langle\bar{x}_i-\Delta,w_i -
    \bar{w}_i\rangle+ \langle-\bar{x}_i, w_{i_U} -
    \bar{w}_{i_U}\rangle.\label{e16}\end{align} \label{lem2}
\end{lemma}
\begin{proof}
  The proof is based on the computation of dual subgradients, e.g.,
  in~\cite{DPB:09,DPB-AN-AO:03a}.
\end{proof}
A direct result of Lemma~\ref{lem2} is that the vector $(g(x_i(k))^T,
(-\Delta - x_i(k))^T, x_i(k)^T, (x_i(k)-\Delta)^T, -x_i(k)^T)$ is a
supergradient of $Q_i$ at $v_i(k)$; i.e., the following supergradient
inequality holds for any $\xi_i\in \Xi'$:
\begin{align}
  &Q_i(\xi_i) - Q_i(v_i(k))\leq \langle g(x_i(k)),\mu_i -
  \mu_i(k)\rangle + \langle -\Delta - x_i(k), \lambda_i -
  v_{\lambda}^i(k)_i\rangle\nnum\\ &+ \langle x_i(k), \lambda_{i_U} -
  v_{\lambda}^i(k)_{i_U}\rangle + \langle x_i(k)-\Delta, w_i -
  v_w^i(k)_i\rangle + \langle -x_i(k),w_{i_U} -
  v_w^i(k)_{i_U}\rangle.\label{e13}
\end{align}

Now we can see that the update rule~\eqref{e48} of dual estimates in the DADS
algorithm is a combination of a dual subgradient scheme and average consensus algorithms. The following establishes that $Q_i$ is Lipschitz continuous with some Lipschitz constant $L$.

\begin{lemma}[Lipschitz continuity of $Q_i$]
  There is a constant $L > 0$ such that for any $\xi_i, \bar{\xi}_i
  \in \Xi'$, it holds that $\|Q_i(\xi_i) -
    Q_i(\bar{\xi}_i)\| \leq L \|\xi_i -
    \bar{\xi}_i\|$.\label{lem6}
\end{lemma}

\begin{proof}
  Similarly to Lemma~\ref{lem2}, one can show that if $\bar{x}_i\in
  \Omega_i(\bar{\xi}_i)$, then $(g(\bar{x}_i)^T, (-\Delta -
  \bar{x}_i)^T, \bar{x}_i^T, (\bar{x}_i-\Delta)^T, -\bar{x}_i^T)^T$ is
  a supergradient of $Q_i$ at $\bar{\xi}_i$; i.e., the following holds
  for any $\xi_i\in\Xi'$:
  \begin{align}
    &Q_i(\xi_i) - Q_i(\bar{\xi}_i)\leq \langle g(\bar{x}_i),\mu_i - \bar{\mu}_i\rangle + \langle -\Delta - \bar{x}_i, \lambda_i - \bar{\lambda}_i\rangle\nnum\\
    &+ \langle \bar{x}_i, \lambda_{i_U} - \bar{\lambda}_{i_U}\rangle + \langle\bar{x}_i-\Delta,w_i - \bar{w}_i\rangle + \langle-\bar{x}_i,w_{i_U} -
    \bar{w}_{i_U}\rangle.\nnum
  \end{align}
  Since $\|g(\bar{x}_i)\| \leq G$ and $\|\bar{x}_i\| \leq H$, there is
  $L > 0$ such that $Q_i(\xi_i) - Q_i(\bar{\xi}_i) \leq L \|\xi_i -
  \bar{\xi}_i\|$. Similarly, $Q_i(\bar{\xi}_i) - Q_i(\xi_i) \leq L
  \|\xi_i - \bar{\xi}_i\|$. We then reach the desired result.
\end{proof}

In the DADS algorithm, the error induced by the projection map $P_{M_i}$ is
given by:
\begin{align}
  &e_i(k) := P_{M_i}[v_i(k) + \alpha(k) \DD_i(k)] - v_i(k).\nnum
\end{align}
We next provide a basic iterate relation of dual estimates in the DADS algorithm.

\begin{lemma} [Basic iterate relation]
  Under the assumptions in Theorem~\ref{the2}, for any
  $((\mu_i),\lambda,w)\in \Xi$ with $(\mu_i,\lambda,w)\in M_i$ for all
  $i\in V$, the following estimate holds for all $k\geq 0$:
  \begin{align}
  &\sum_{i\in V}\|e_i(k) - \alpha(k){\DD}_i(k)\|^2\leq
    \alpha(k)^2\sum_{i\in V}\|{\DD}_i(k)\|^2+\sum_{i\in
      V}(\|\xi_i(k)-\xi_i\|^2 -
    \|\xi_i(k+1)-\xi_i\|^2)\nnum\\ &+2\alpha(k)\sum_{i\in V}\{\langle
    g(x_i(k)),\mu_i(k)-\mu_i\rangle +\langle -\Delta - x_i(k),
    v_{\lambda}^i(k)_i-\lambda_i\rangle\nnum\\
 &+\langle x_i(k),
    v_{\lambda}^i(k)_{i_U}-\lambda_{i_U}\rangle+\langle x_i(k)-\Delta,
    v_w^i(k)_i-w_i\rangle + \langle -x_i(k),
    v_w^i(k)_{i_U}-w_{i_U}\rangle\}.\label{e26}
\end{align}
\label{lem4}\end{lemma}
\begin{proof}
  Recall that $M_i$ is closed and convex. The proof is a combination
  of the nonexpansion property of projection operators
  in~\cite{DPB-AN-AO:03a} and the property of balanced graphs.
\end{proof}
The lemma below shows the asymptotic convergence of dual estimates.
\begin{lemma}[Dual estimate convergence] Under the assumptions in
  Theorem~\ref{the2}, there exists a dual optimal solution $\xi^* :=
  ((\mu_i^*),\lambda^*,w^*)\in D_{\Delta}^*$ such that
  $\displaystyle{\lim_{k\rightarrow+\infty}\|\mu_i(k) -
    \mu_i^*\| = 0}$,
  $\displaystyle{\lim_{k\rightarrow+\infty}\|\lambda^i(k) -
    \lambda^*\| = 0}$, and
  $\displaystyle{\lim_{k\rightarrow+\infty}\|w^i(k) - w^*\| =
    0}$.\label{lem1}
\end{lemma}

\begin{proof} By the dual decomposition property~\eqref{e11} and the
  boundedness of dual optimal solution sets, the dual problem
  $(D_{\Delta})$ is equivalent to the following:
  \begin{align}
    \max_{(\xi_i)} \sum_{i\in V}Q_i(\xi_i),\quad {\rm s.t.}\quad \xi_i
    \in M_i.\label{e15}
\end{align}
Note that $Q_i$ is affine and $M_i$ is convex, implying that the
problem~\eqref{e15} is a constrained convex programming where the
global objective function is a simple sum of local ones and the local
state constraints are convex and compact. The rest of the proofs can
be finished by following similar lines in~\cite{MZ-SM:09c}, and thus
omitted.\end{proof}

The remainder of this section is dedicated to characterizing the
convergence properties of primal estimates. Toward this end, we
present some properties of $\Omega_i$.

\begin{lemma}[Properties of marginal maps]
  The set-valued marginal map $\Omega_i$ is
  closed. In addition, it is upper semicontinuous at $\xi_i\in \Xi'$;
  i.e., for any $\epsilon' > 0$, there is $\delta' > 0$ such that for
  any $\tilde{\xi}_i\in B_{\Xi'}(\xi_i,\delta')$, it holds that
  $\Omega_i(\tilde{\xi}_i)\subset
  B_{2^X}(\Omega_i(\xi_i),\epsilon')$.\label{lem3}
\end{lemma}
\begin{proof} Consider sequences $\{x_i(k)\}$ and $\{\xi_i(k)\}$
  satisfying $\displaystyle{\lim_{k\rightarrow+\infty}\xi_i(k) =
    \bar{\xi}_i}$, $x_i(k)\in\Omega_i(\xi_i(k))$ and
  $\displaystyle{\lim_{k\rightarrow+\infty}x_i(k) = \bar{x}_i}$. Since
  $\LL_i$ is continuous, then we
  have \begin{align*}
    \LL_i(\bar{x}_i,\bar{\xi}_i)
    = \lim_{k\rightarrow+\infty}\LL_i(x_i(k),\xi_i(k))
    \leq\lim_{k\rightarrow+\infty} (Q_i(\xi_i(k))) =
    Q_i(\bar{\xi}_i),
  \end{align*}
  where in the inequality we use the property of
  $x_i(k)\in\Omega_i(\xi_i(k))$, and in the last equality
  we use the continuity of $Q_i$. Then
  $\bar{x}_i\in\Omega_i(\bar{\xi}_i)$ and the closedness of
  $\Omega_i$ follows.

  Note that $\Omega_i(\xi_i) =
  \Omega_i(\xi_i)\cap X$. Recall that $\Omega_i$
  is closed and $X$ is compact. Then it is a result of
  Proposition 1.4.9 in~\cite{JPA-HF:90} that $\Omega_i(\xi_i)$ is upper
  semicontinuous at $\xi_i\in \Xi'$; i.e, for any neighborhood
  ${\mathcal{U}}$ in $2^X$ of $\Omega_i(\xi_i)$, there is
  $\delta' > 0$ such that $\forall \tilde{\xi}_i\in
  B_{\Xi'}(\xi_i,\delta')$, it holds that
  $\Omega_i(\tilde{\xi}_i)\subset {\mathcal{U}}$. Let
  ${\mathcal{U}} = B_{2^{X}}(\Omega_i(\xi_i), \epsilon')$,
  and we obtain upper semicontinuity at $\xi_i$.
\end{proof}
With the above results, one can show the convergence of primal
estimates.
\begin{lemma}[Primal estimate convergence]
Under the assumptions in Theorem~\ref{the2}, for each $i\in V$, it
holds that
$\displaystyle{\lim_{k\rightarrow+\infty}x_i(k)=\tilde{x}_i}$ where
$\tilde{x}_i= \Omega_i(\xi_i^*)$.\label{lem3}
\end{lemma}

\begin{proof} The combination of upper semicontinuity of $\Omega_i$
in Lemma~\ref{lem3} and
$\displaystyle{\lim_{k\rightarrow+\infty}\xi_i(k)=\xi_i^*}$ with
$\xi_i^*$ given in Lemma~\ref{lem1} ensures that each accumulation
point of $\{x_i(k)\}$ is a point in the set $\Omega_i(\xi_i^*)$;
i.e., the convergence of $\{x_i(k)\}$ to the set $\Omega_i(\xi_i^*)$
can be guaranteed. Since $\Omega_i(\xi_i^*)$ is singleton, then
$\tilde{x}_i = \Omega_i(\xi_i^*)$. We arrive in the desired
result. \end{proof}

Now we are ready to show the main result of this paper,
Theorem~\ref{the2}. In particular, we will show complementary
slackness, primal feasibility of $\tilde{x}$, and its primal
optimality, respectively.

\textbf{Proof for Theorem~\ref{the2}:}

\textbf{Claim 1:} $\langle-\Delta - \tilde{x}_i + \tilde{x}_{i_D},
{\lambda}_i^*\rangle = 0$, $\langle -\Delta + \tilde{x}_i -
\tilde{x}_{i_D}, {w}_i^*\rangle = 0$ and $\langle
g(\tilde{x}_i),{\mu}_i^*\rangle = 0$.

\begin{proof}
Rearranging the terms related to $\lambda$ in~\eqref{e26} leads to the
following inequality holding for any $((\mu_i),\lambda,w)\in \Xi$ with
$(\mu_i,\lambda,w)\in M$ for all $i\in V$:
  \begin{align}
    &-\sum_{i\in V}2\alpha(k)(\langle -\Delta - x_i(k),
    v_{\lambda}^i(k)_i-\lambda_i\rangle + \langle x_{i_D}(k),
    v_{\lambda}^{i_D}(k)_i-\lambda_i\rangle)\nnum\\ &\leq
    \alpha(k)^2\sum_{i\in V}\|{\DD}_i(k)\|^2+\sum_{i\in
      V}(\|\xi_i(k)-\xi_i\|^2 -
    \|\xi_i(k+1)-\xi_i\|^2) \label{e34}\\ &+2\alpha(k)\sum_{i\in V}\{\langle
    -x_i(k), v_w^i(k)_{i_U}-w_{i_U}\rangle + \langle
    x_i(k)-\Delta,v_w^i(k)_i-w_i\rangle +\langle
    g(x_i(k)),\mu_i(k)-\mu_i\rangle\}.\nnum
  \end{align}
Sum~\eqref{e34} over $[0,K]$, divide by $s(K):=\sum_{k=0}^K\alpha(k)$,
and we have
\begin{align}
  &\frac{1}{s(K)}\sum_{k=0}^{K}\alpha(k)\sum_{i\in V}2(\langle \Delta
  + x_i(k), v_{\lambda}^i(k)_i-\lambda_i\rangle+ \langle -x_{i_D}(k),
  v_{\lambda}^{i_D}(k)_i-\lambda_i\rangle)\le \nnum\\
 &\leq
  \frac{1}{s(K)}\sum_{k=0}^{K}\alpha(k)^2\sum_{i\in V}\|{\DD}_i(k)\|^2
  +\frac{1}{s(K)}\{\sum_{i\in V}(\|\xi_i(0)-\xi_i\|^2 -
  \|\xi_i(K+1)-\xi_i\|^2) \label{e28}\\
 &+\sum_{k=0}^{K}2\alpha(k)\sum_{i\in
    V}(\langle g(x_i(k)),\mu_i(k)-\mu_i\rangle+\langle
  x_i(k)-\Delta,v_w^i(k)_i-w_i\rangle + \langle -x_i(k),
  v_w^i(k)_{i_U}-w_{i_U}\rangle)\}.
\end{align}
We now proceed to show $\langle-\Delta - \tilde{x}_i +
\tilde{x}_{i_D}, {\lambda}_i^*\rangle \geq 0$ for each $i\in
V$. Notice that we have shown that
$\displaystyle{\lim_{k\rightarrow+\infty}\|x_i(k)-\tilde{x}_i\|=0}$
for all $i\in V$, and it also holds that
$\displaystyle{\lim_{k\rightarrow+\infty}\|\xi_i(k)-\xi_i^*\|=0}$ for
all $i\in V$. Let $\lambda_i = \frac{1}{2}{\lambda}_i^*$, $\lambda_j =
{\lambda}_j^*$ for $j\neq i$ and $\mu_i = {\mu}_i^*$, $w = {w}^*$
in~\eqref{e28}. Recall that $\{\alpha(k)\}$ is not summable but square
summable, and $\{{\DD}_i(k)\}$ is uniformly bounded. Take
$K\rightarrow+\infty$, and then it follows from Lemma 5.1
in~\cite{MZ-SM:09c} that:
\begin{align}
  &\langle \Delta + \tilde{x}_i - \tilde{x}_{i_D},
  {\lambda}_i^*\rangle \leq 0.\label{e30}
\end{align}
On the other hand, since $\xi^*\in D^*_{\Delta}$, we have
$\|\xi^*\|\leq \gamma$ given the fact that $\gamma$ is an upper bound
of $D_{\Delta}^*$. Let $\xi := (\mu,\lambda,w)$ where $\xi_i :=
(\mu_i,\lambda,w)$. Then we could choose a sufficiently small $\delta'
> 0$ and $\xi\in \Xi$ in~\eqref{e28} such that $\|\xi_i\| \leq \gamma
+ \theta_i$ where $\theta_i$ is given in the definition of $M_i$ and
$\xi$ is given by: $\lambda_i = (1+\delta'){\lambda}_i^*$, $\lambda_j
= {\lambda}_j^*$ for $j\neq i$, $w = w^*$, $\mu = \mu^*$. Following
the same lines toward~\eqref{e30}, it gives that $-\delta \langle
\Delta + \tilde{x}_i - \tilde{x}_{i_D}, {\lambda}_i^*\rangle \leq
0$. Hence, it holds that $\langle-\Delta - \tilde{x}_i +
\tilde{x}_{i_D}, {\lambda}_i^*\rangle = 0$. The rest of the proof is
analogous and thus omitted.
\end{proof}
The proofs of the following two claims are in part analogous and can be
found in~\cite{MZ-SM:10-arxiv}.

\textbf{Claim 2:} $\tilde{x}$ is primal feasible to the approximate
problem ($P_{\Delta}$).

\begin{proof} We have known that $\tilde{x}_i\in X$. We
  proceed to show $-\Delta - \tilde{x}_i + \tilde{x}_{i_D} \leq 0$ by
  contradiction. Since $\|\xi^*\|\leq \gamma$, we could choose a sufficiently small $\delta' > 0$ and $\xi := (\mu,\lambda,w)$ where $\xi_i := (\mu_i,\lambda,w)$ and $\|\xi_i\| \leq \gamma + \theta_i$ in~\eqref{e28} as follows: if $(-\Delta -
  \tilde{x}_i + \tilde{x}_{i_D})_{\ell} > 0$, then $(\lambda_i)_{\ell}
  = ({\lambda}_i^*)_{\ell} + \delta'$; otherwise,
  $(\lambda_i)_{\ell} = ({\lambda}_i^*)_{\ell}$, and $w = w^*$, $\mu = \mu^*$. The rest of the proofs is analogous to Claim 1.

  Similarly, one can show $g(\tilde{x}_i)\leq0$ and $-\Delta +
  \tilde{x}_i - \tilde{x}_{i_D} \leq 0$ by applying analogous
  arguments. We conclude that $\tilde{x}$ is primal feasible to the approximate problem ($P_{\Delta}$).
\end{proof}

\textbf{Claim 3:} $\tilde{x}$ is a primal solution to the problem ($P_{\Delta}$).

\begin{proof}
  Since $\tilde{x}$ is primal feasible to the approximate problem ($P_{\Delta}$), then $\sum_{i\in
    V}f_i(\tilde{x}_i) \geq p^*_{\Delta}$. On the other hand, it follows from Claim 1 that
  \begin{align*}\sum_{i\in V}f_i(\tilde{x}_i) = \sum_{i\in
    V}\LL_i(\tilde{x}_i,\xi_i^*) \leq \sum_{i\in
    V}Q_i(\xi_i^*) \leq p^*_{\Delta}.\end{align*} We then conclude that $\sum_{i\in
    V}f_i(\tilde{x}_i) = p^*_{\Delta}$. In conjunction with the feasibility of $\tilde{x}$, this further 
\end{proof}

\section{Simulations}

In the extended version~\cite{MZ-SM:10-arxiv}, we examine several
numerical examples to illustrate the performance of our algorithm.
These present different cases of a robust source localization example,
where (i) the SD assumption is satisfied, (ii) the SD assumption is
violated, and (iii) a comparison with gradient-based algorithms is
made. An additional example includes that of a non-convex quadratic
program for which properties P1 and P2 can easily be verified.

\subsection{Robust source localization}

We consider a robust source localization problem where the objective
function is adopted from~\cite{AA-FB-LG-CS:08,POE-JG-JX-PO:10}. In
particular, consider a network of four agents $V \triangleq
\{1,\cdots,4\}$. The objective functions of agents are piecewise
linear and given by $f_i(z) = |\|z - a_i\| - r|$. The local inequality
functions are given by: \begin{align*}g_1(z) = \left[\begin{array}{c}
      z_1 - 8 \\
      -z_1 - 8 \\
      z_2 - 8 \\
      -z_2 - 8 \\
    \end{array}
  \right]
,\quad g_2(z) = \left[
    \begin{array}{c}
      z_1 - 9 \\
      -z_1 - 9 \\
      z_2 - 9 \\
      -z_2 - 9 \\
    \end{array}
  \right],\quad g_3(z) = \left[\begin{array}{c}
      z_1 - 8.5 \\
      -z_1 - 8.5 \\
      z_2 - 8.5 \\
      -z_2 - 8.5 \\
    \end{array}
  \right],\quad g_4(z) = \left[
    \begin{array}{c}
      z_1 - 9.5 \\
      -z_1 - 9.5 \\
      z_2 - 9.5 \\
      -z_2 - 9.5 \\
    \end{array}
  \right],\end{align*} and, the local constraint sets are given by \begin{align*}&X_1 = \{z\in\real^2\;|\; -10\leq z_1 \leq 10,\;\;-10\leq z_2 \leq 10\},\\
  &X_2 = \{z\in\real^2\;|\; -10.5\leq z_1 \leq 10.5,\;\;-10.5\leq z_2 \leq 10.5\},\\
  &X_3 = \{z\in\real^2\;|\; -9\leq z_1 \leq 9,\;\;-10\leq z_2 \leq 10\},\\
  &X_4 = \{z\in\real^2\;|\; -11\leq z_1 \leq 11,\;\;-9\leq z_2 \leq 9\}.\end{align*}

In the simulation, we choose the parameter $\delta = 0.1$. The local
Lagrangian function can be written as $\LL_i(x_i,\xi_i) = f_i(x_i) +
\langle\zeta_i,x_i\rangle$ by dropping the terms independent of $x_i$
and $\zeta_i$ is linear in $\xi_i$. Figure~\ref{fig_section} shows the
sectional plot of $f(z) \triangleq \sum_{i\in V}f_i(z)$ along
$z_1$-axle, demonstrating that $f$ is nonconvex and has local minima.

The inter-agent topologies $\GG(k)$ are given by: $\GG(k)$ is $1\leftrightarrow2\leftrightarrow3\leftrightarrow4$ when $k$ is odd, and $\GG(k)$ is $1\rightarrow2\leftrightarrow3\leftarrow4\rightarrow1$ when $k$ is even. It is easy to see that $\GG(k)$ satisfies the periodical strong connectivity assumption~\ref{asm10}.

\subsubsection{Simulation 1; the assumption of SD is satisfied}

For this numerical simulation, we consider the set of parameters $r =
0.75$, $a_1 = [0\;\; 0]^T$, $a_2 = [0\;\; 1]^T$, $a_3 = [1\;\; 0]^T$
and $a_4 = [1\;\; 1]^T$. Figure~\ref{fig_surf} shows the surface of the
global objective function $f(x) = \sum_{i\in V}f_i(x)$. The contour,
Figure~\ref{fig_contour}, indicates that the set of optimal solutions
is a region around $[0.5 \;\; 0.5]^T$. Figure~\ref{fig_objective} is
the sectional plot of $f_1$ along $z_1$-axle.

From Figures~\ref{fig_slop1}--\ref{fig_slop2_portion}, one can see
that $\zeta_i(k)$ converges to some point
$(0,0.05]\times(0,0.05]$. Hence, $\xi^*\in D_s^*$; i.e., the
assumption of SD is satisfied.

The simulation results are shown in Figures~\ref{fig_x1}
to~\ref{fig_x2}. In particular, Figure~\ref{fig_x1}
(resp. Figure~\ref{fig_x2}) shows the evolution of primal estimates of
the primal solution $x^*(1)$ (resp. $x^*(2)$). After about 25
iterates, the primal estimates oscillate within a very small region
and eventually agree upon the point $[0.4697\;0.472]^T$ which coincides
with a global optimal solution.

\subsubsection{Simulation 2; the assumption of SD is violated}

Consider the same problem as Simulation 1 with $r = 0.75$ and $a_i =
[0\;\; 0]^T$ for $i\in V$. From Figures~\ref{fig_slop1_oscillation}
and~\ref{fig_slop2_oscillation}, one can see that $\zeta_i(k)$
converges to $[0\;\;0]^T$. Hence, $\xi^*\notin D_s^*$ and the
assumption of SD is not satisfied. Figures~\ref{fig_x1_oscillation}
and~\ref{fig_x2_oscillation} confirms that primal estimates fail to
converge in this case.

\subsubsection{Simulation 3; comparison with gradient-based algorithms}

Consider the same set of parameters as in Simulation~1 without
including the inequality constraints. The multi-agent interaction
topologies are the same. We implement the diffusion gradient algorithm
in~\cite{AN-AO-PAP:08} for this problem. Figures~\ref{fig_grad_x1}
and~\ref{fig_grad_x2} show that the primal estimates reach the
consensus value of $[-0.65\;\;-0.38]^T$ after $40000$ iterates. From
Figure~\ref{fig_contour}, it is clear that $[-0.65\;\;-0.38]^T$ is not a
global optimum. By comparing Figures~\ref{fig_x1},
~\ref{fig_x2},~\ref{fig_grad_x1} and~\ref{fig_grad_x2}, one can see
that our algorithm is much faster than the diffusion gradient method
at the expense of solving a global optimization problem at each
iterate.

We also implement the incremental gradient algorithm
in~\cite{SSR-AN-VVV:08b} for the same set of parameters in Simulation
1 without including inequality
constraints. Figure~\ref{fig_incremental} demonstrates that the
performance of the incremental gradient method is analogous to the
diffusion gradient algorithm; i.e., the estimates are trapped in some
local minimum, and the convergence rate is slower than our algorithm.

\subsection{Nonconvex quadratic programming}

Consider a network of four agents where the topologies are the same as before.
The local objective function is $f_i(z) = \|z\|^2_{P_i} + \langle q_i,z\rangle$ and the
local constraint function is $g_i(z) = \|z\|^2_{A_i} + \langle
b_i,z\rangle + c_i \leq 0$. In particular, we use the following
parameters:
\begin{align*}&P_1 = P_2 = P_3 = \left[
\begin{array}{cc}
0 & 1 \\
1 & 1 \\
\end{array}
\right],\quad P_4 = \left[
\begin{array}{cc}
0 & 1 \\
1 & 0 \\
\end{array}
\right],\\
& A_1 = A_4 = \left[
\begin{array}{cc}
18 & 0 \\
0 & 8 \\
\end{array}
\right],\quad b_1 = b_4 = [2\;\;0],\quad c_1 = c_4 = -1,\\
& A_2 = \left[
\begin{array}{cc}
13 & -2 \\
-2 & 8 \\
\end{array}
\right],\quad b_2 = [0\;\;4],\quad c_2 = -1,\\
& A_3 = \left[
\begin{array}{cc}
5 & -5 \\
-5 & 5 \\
\end{array}
\right],\quad b_3 = [10\;\;10],\quad c_3 = -1.
\end{align*} And the local constraint sets are given by \begin{align*}&X_1 = \{z\in\real^2\;|\; -10\leq z_1 \leq 10,\;\;-10\leq z_2 \leq 10\},\\
  &X_2 = \{z\in\real^2\;|\; -10.5\leq z_1 \leq 10.5,\;\;-10.5\leq z_2 \leq 10.5\},\\
  &X_3 = \{z\in\real^2\;|\; -9\leq z_1 \leq 9,\;\;-10\leq z_2 \leq 10\},\\
  &X_4 = \{z\in\real^2\;|\; -11\leq z_1 \leq 11,\;\;-9\leq z_2 \leq 9\}.\end{align*}

One can see that the sum of $P_i$ is
\begin{align*}P = \sum_{i=1}^4P_i = \left[\begin{array}{cc}
0 & 4 \\
4 & 3 \\
\end{array}\right]\end{align*} which is indefinite. We choose $\delta = 0.3$ for the simulation.

The dual estimates associated with the inequality constraints converge
to $\mu_1^* = 0.5027$, $\mu_2^* = 3.1061$, $\mu_3^* = 1.8792$ and
$\mu_4^* = 2.2910$ in Figure~\ref{fig_QP_mu1}. One can verify
that properties P1 and P2 hold in this case:
\begin{align*}
P_i + \mu_i^*A_i > 0,\quad (P_i +
  \mu_i^*A_i)^{-1}\zeta_i^*\in X_i, \quad i\in V.
\end{align*}
The primal estimates converge to $[-0.1933\;\;-0.3005]^T$,
$[-0.2621\;\;-0.5360]^T$, $[-0.1013\;\;-0.0116]^T$ and
$[-0.2144\;\;-0.2667]^T$ in Figures~\ref{fig_QP_x1}
and~\ref{fig_QP_x2}, and the collection of these points consists of a
global optimal solution to the approximate problem.

\section{Conclusions}

We have studied a distributed dual algorithm for a class of
multi-agent nonconvex optimization problems. The convergence of the
algorithm has been proven under the assumptions that (i) the Slater's
condition holds; (ii) the optimal solution set of the dual limit is
singleton; (iii) the network topologies are strongly connected over
any given bounded period. An open question is how to address the
shortcomings imposed by nonconvexity and multi-agent interactions
settings.



\begin{figure}[h]
  \centering
  \includegraphics[width = .5\linewidth]{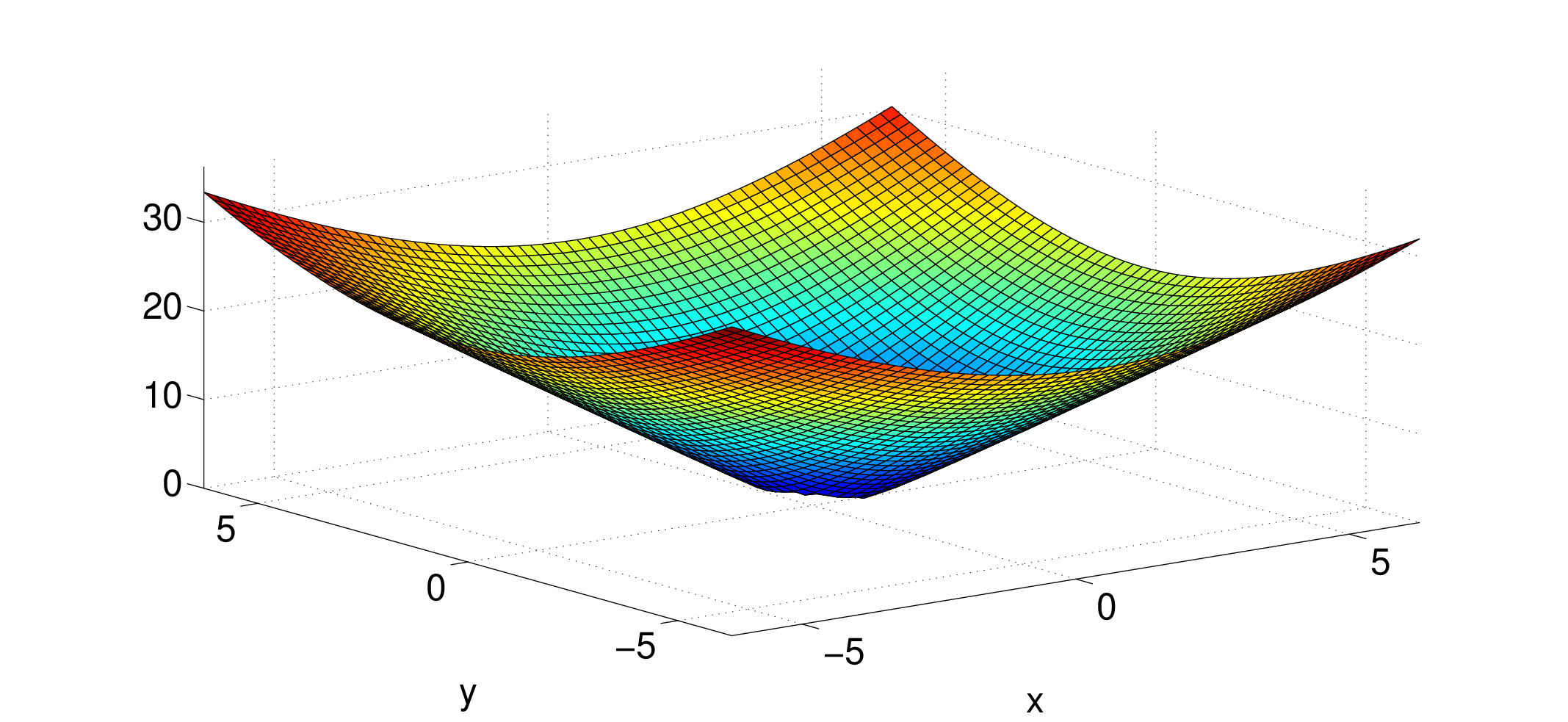}
  \caption{The $3-D$ plot of the global objective function} \label{fig_surf}
  \includegraphics[width = .5\linewidth]{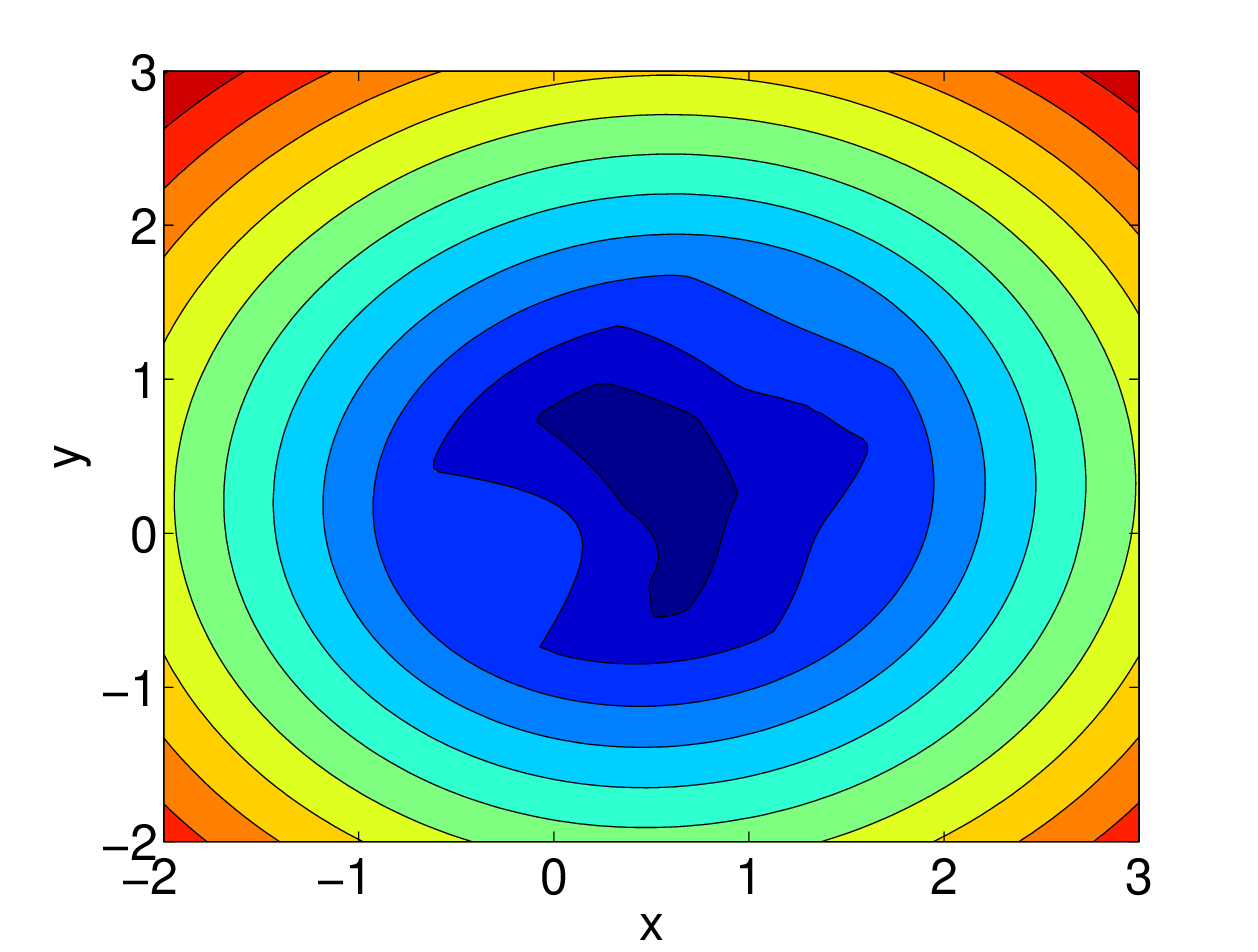}
  \caption{The contour of the global objective function} \label{fig_contour}
\end{figure}

\begin{figure}[h]
  \centering
  \includegraphics[width = .5\linewidth]{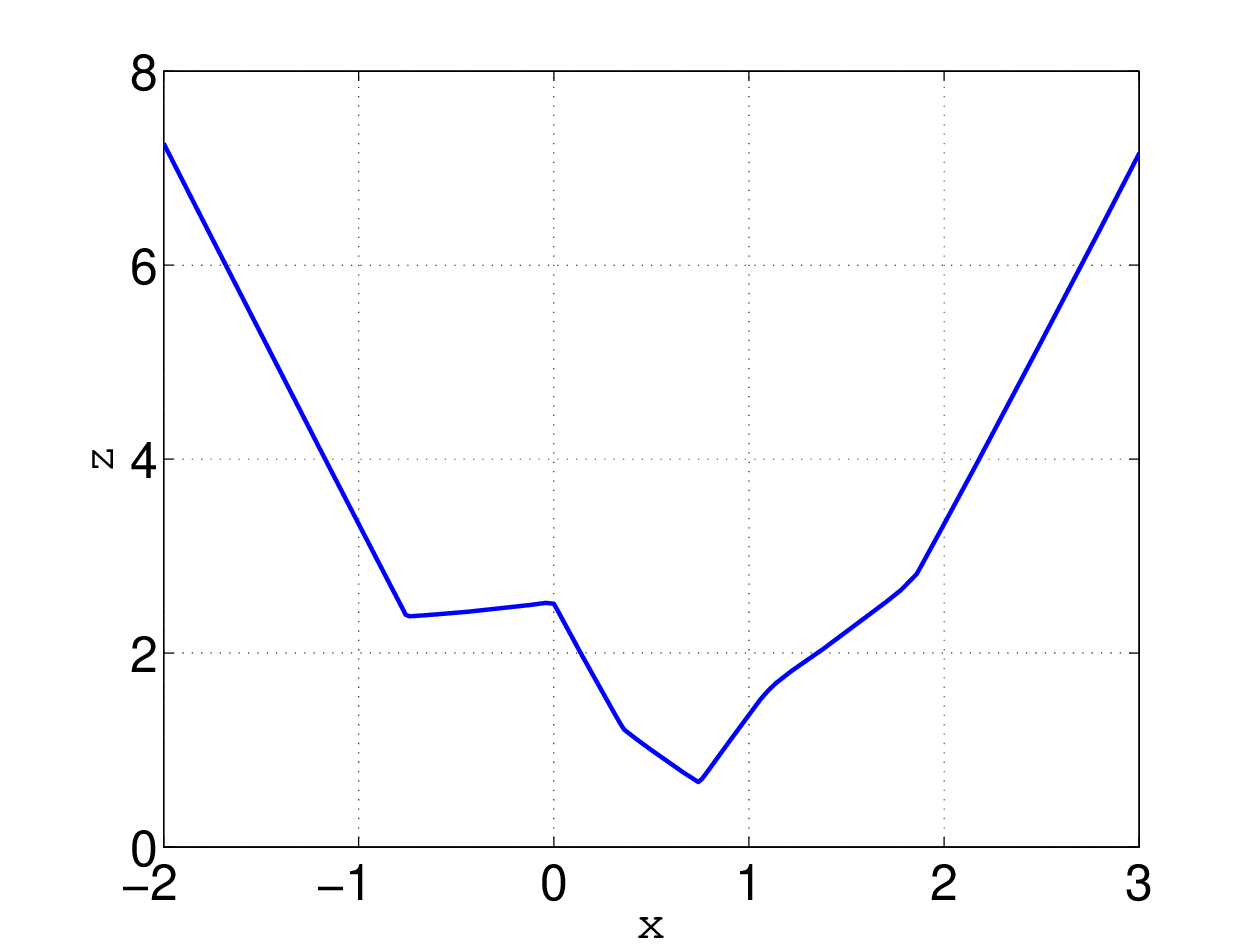}
  \caption{The sectional plot of the global objective function along $z_1$-axle} \label{fig_section}
  \includegraphics[width = .5\linewidth]{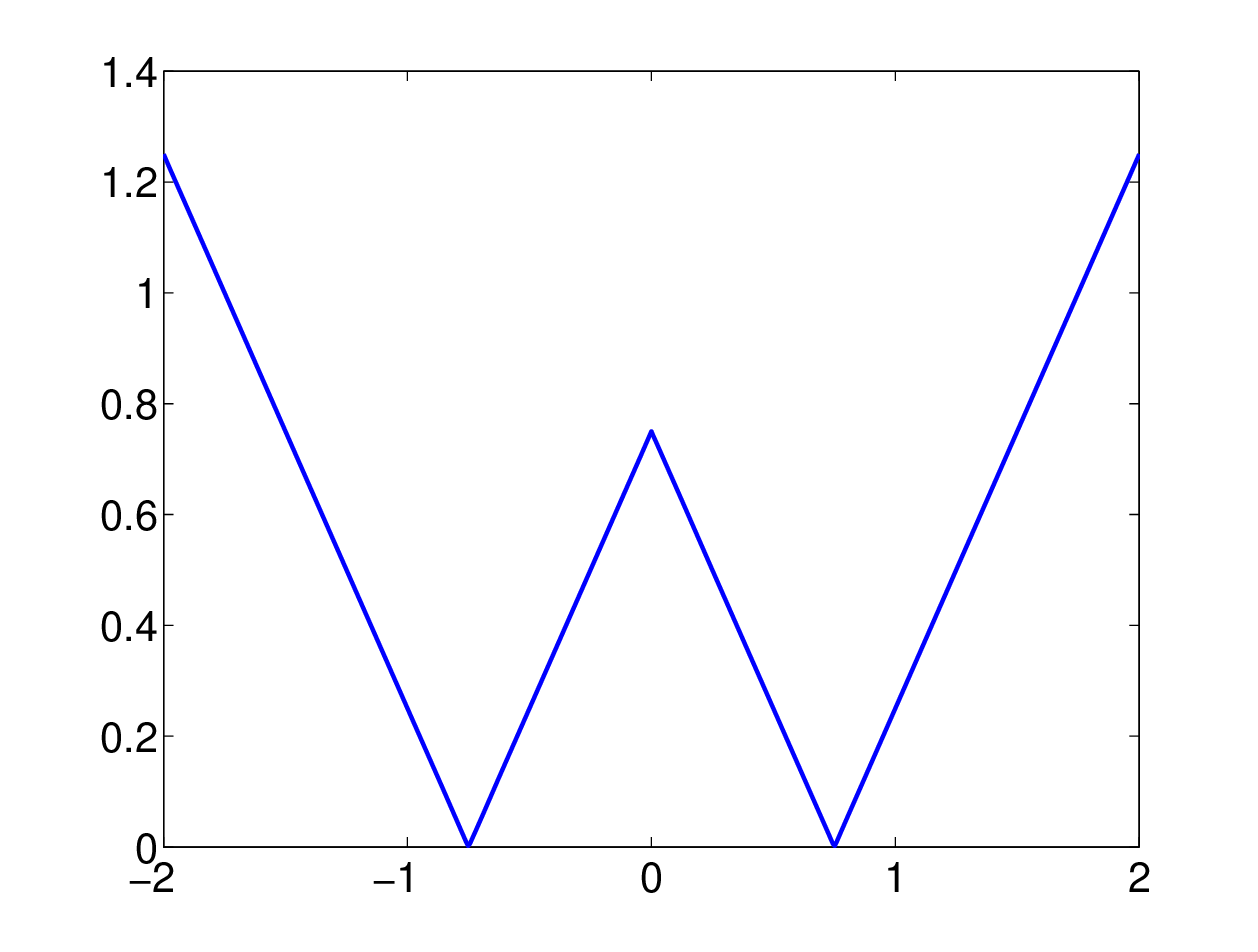}
  \caption{The sectional plot of $f_1$ along $z_1$-axle} \label{fig_objective}
\end{figure}

\begin{figure}[h]
  \centering
  \includegraphics[width = .75\linewidth]{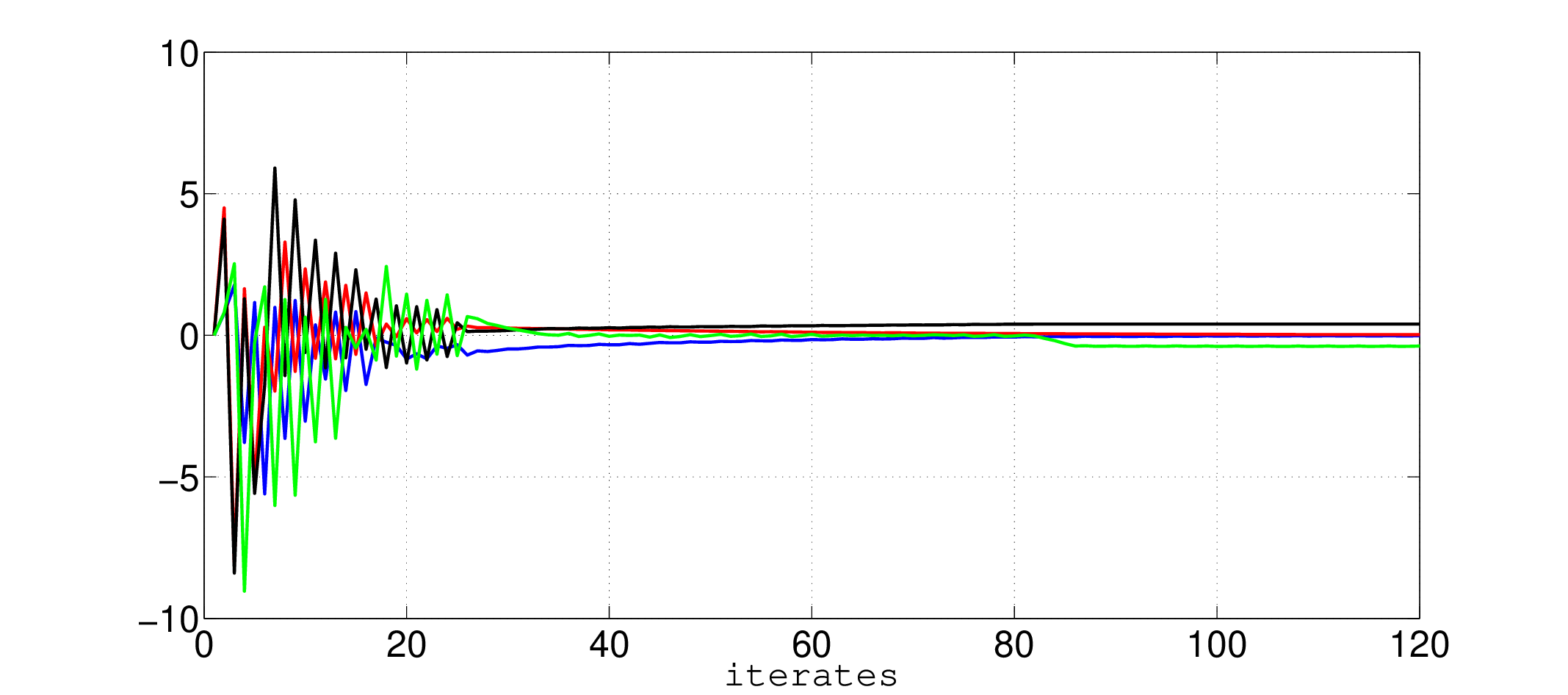}
  \caption{The evolution of $\zeta_{i,1}$} \label{fig_slop1}
  \includegraphics[width = .75\linewidth]{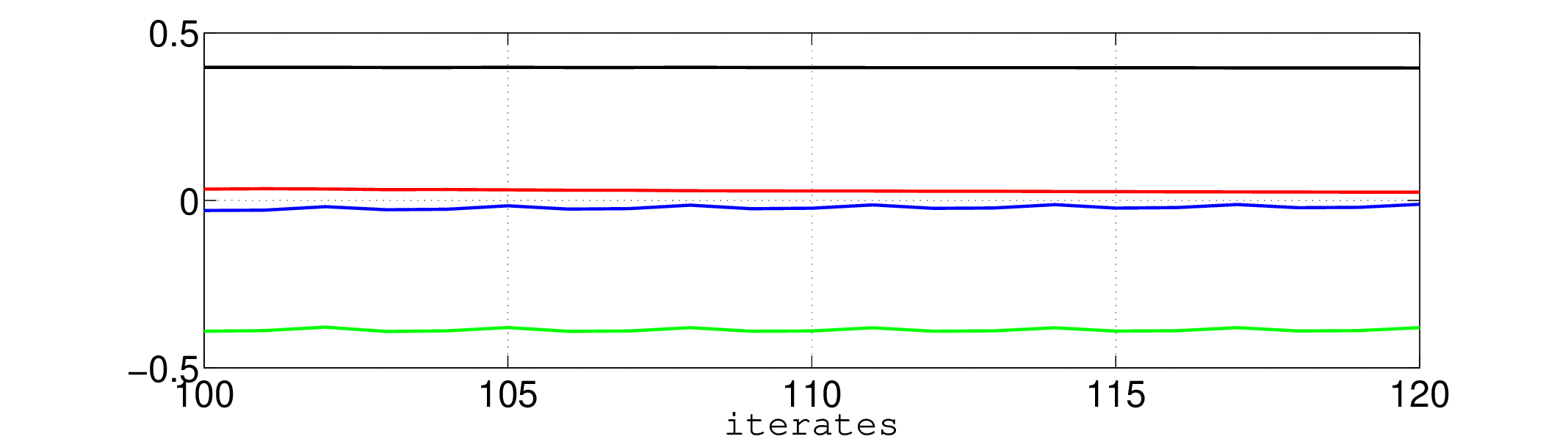}
  \caption{A portion of the evolution of $\zeta_{i,1}$} \label{fig_slop1_portion}
\end{figure}

\begin{figure}[h]
  \centering
  \includegraphics[width = .75\linewidth]{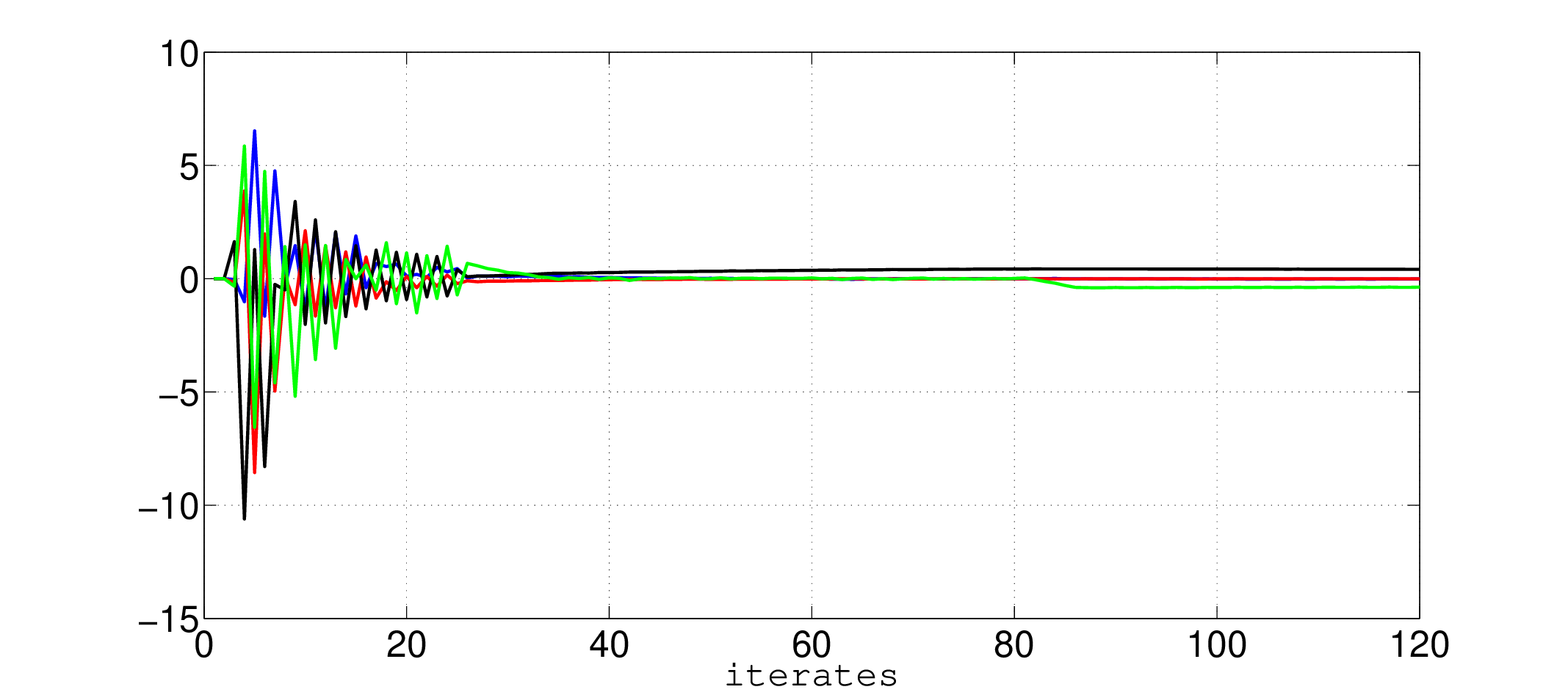}
  \caption{The evolution of $\zeta_{i,2}$} \label{fig_slop2}
  \includegraphics[width = .75\linewidth]{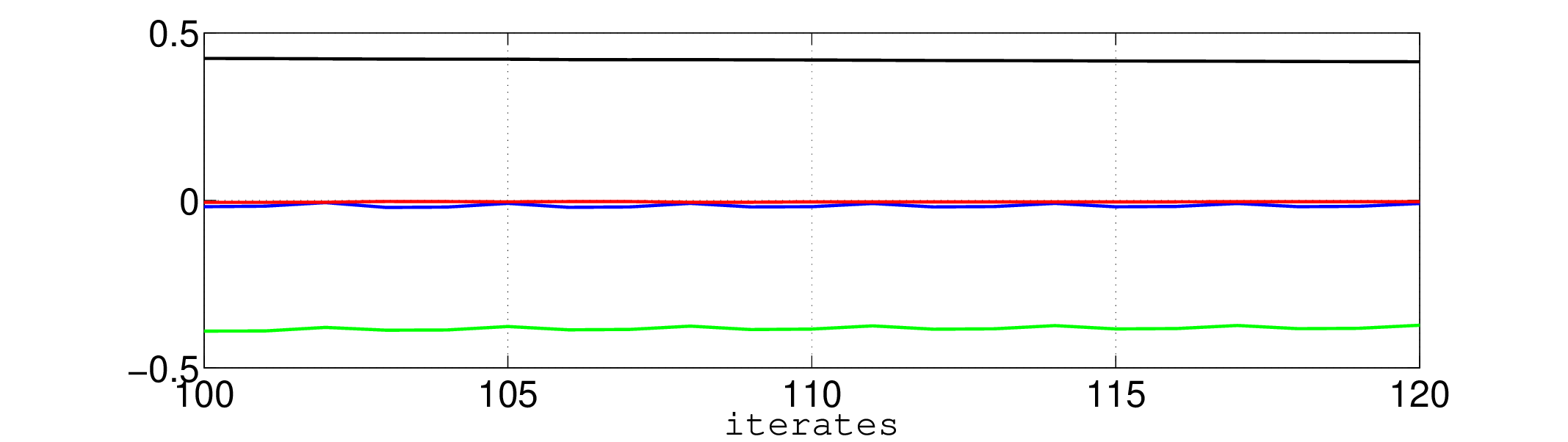}
  \caption{A portion of the evolution of $\zeta_{i,2}$} \label{fig_slop2_portion}
\end{figure}

\begin{figure}[h]
  \centering
  \includegraphics[width = .75\linewidth]{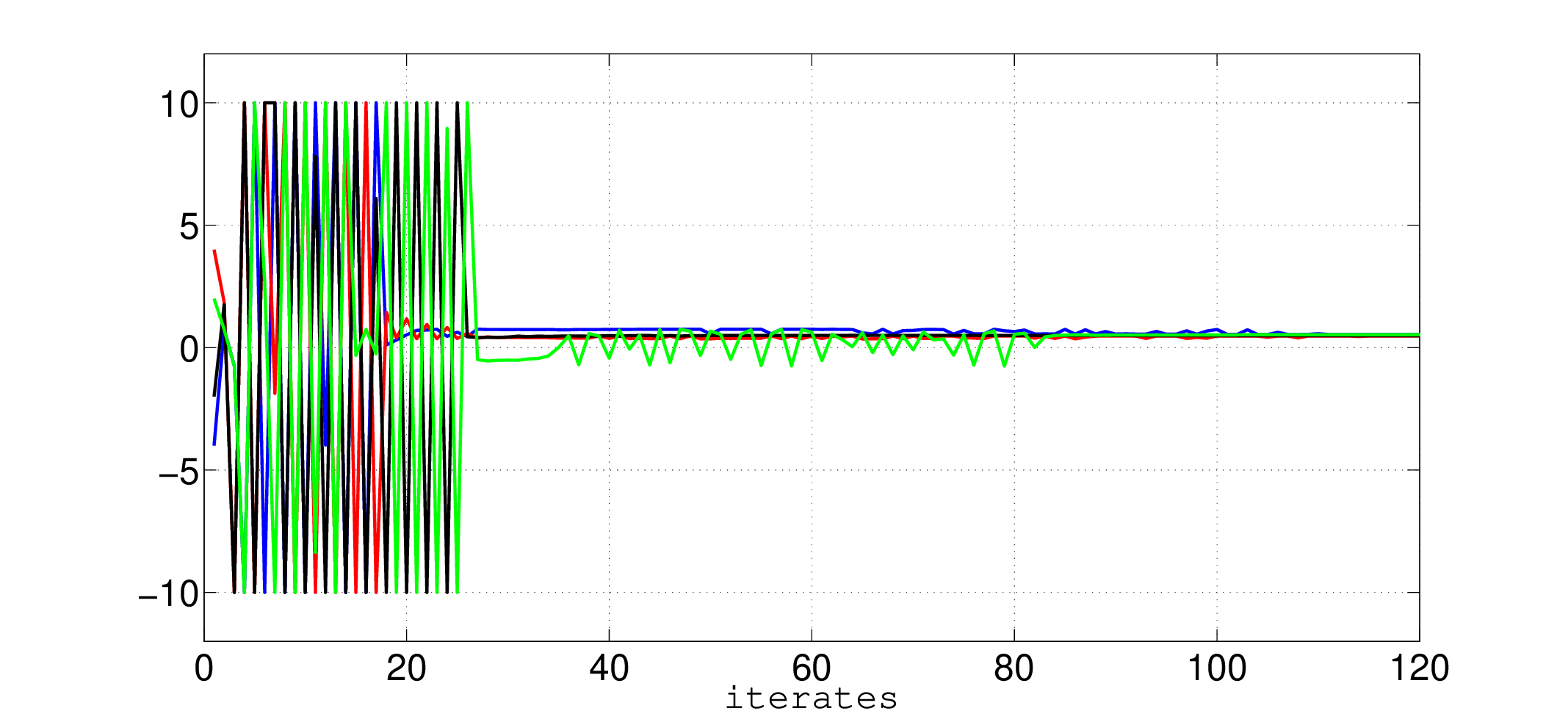}
  \caption{The primal estimates of $x^*(1)$} \label{fig_x1}
  \includegraphics[width = .75\linewidth]{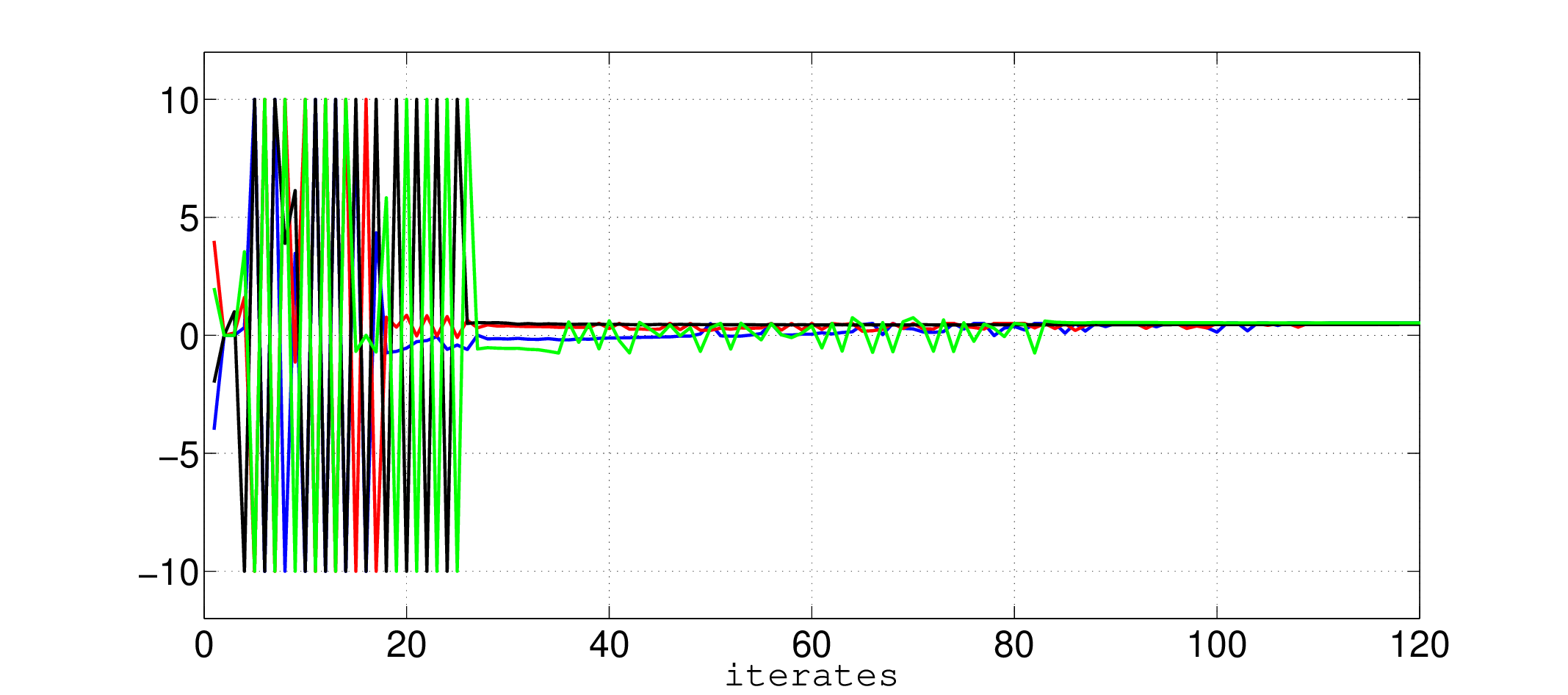}
  \caption{The primal estimates of $x^*(2)$} \label{fig_x2}
\end{figure}


\begin{figure}[h]
  \centering
  \includegraphics[width = .5\linewidth]{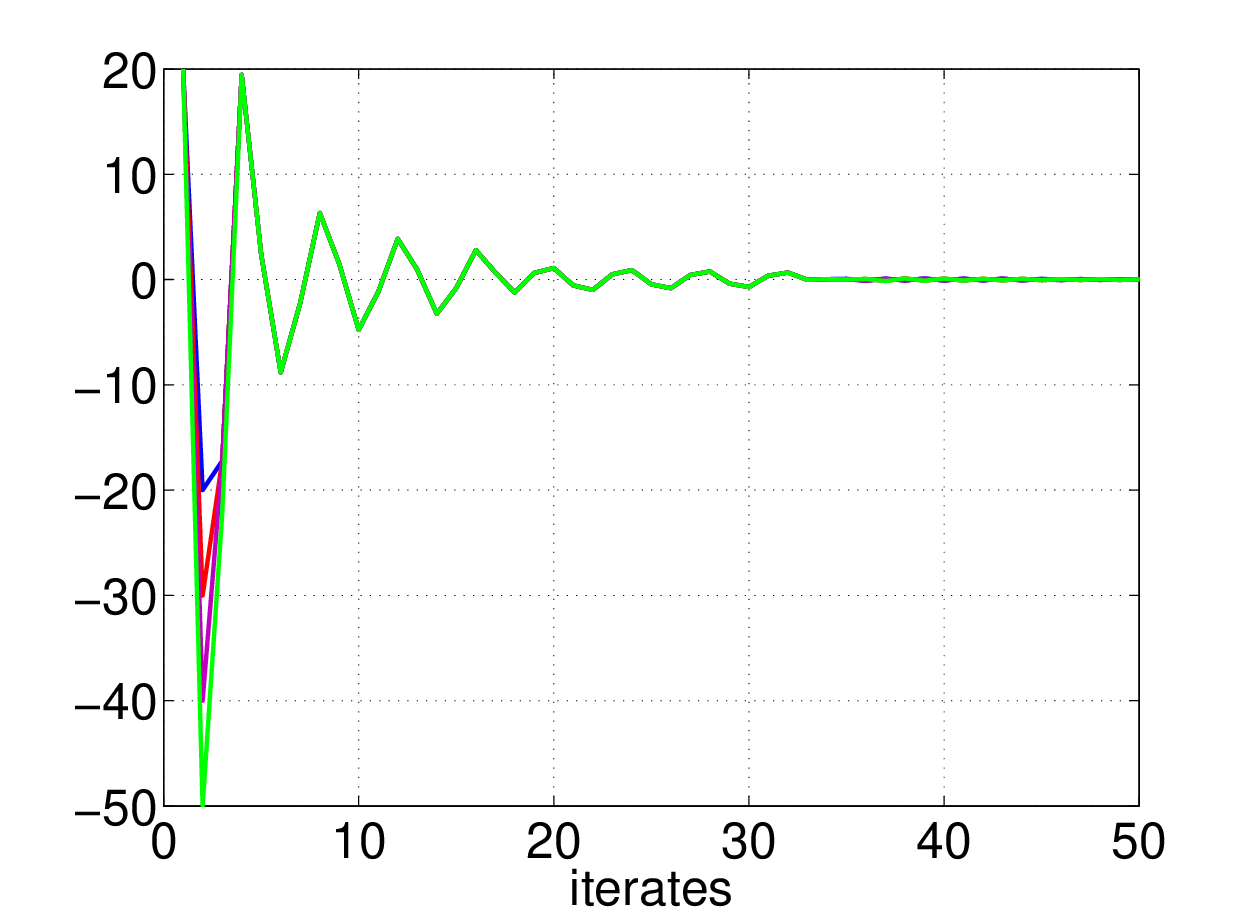}
  \caption{The evolution of $\zeta_{i,1}$} \label{fig_slop1_oscillation}
  \includegraphics[width = .5\linewidth]{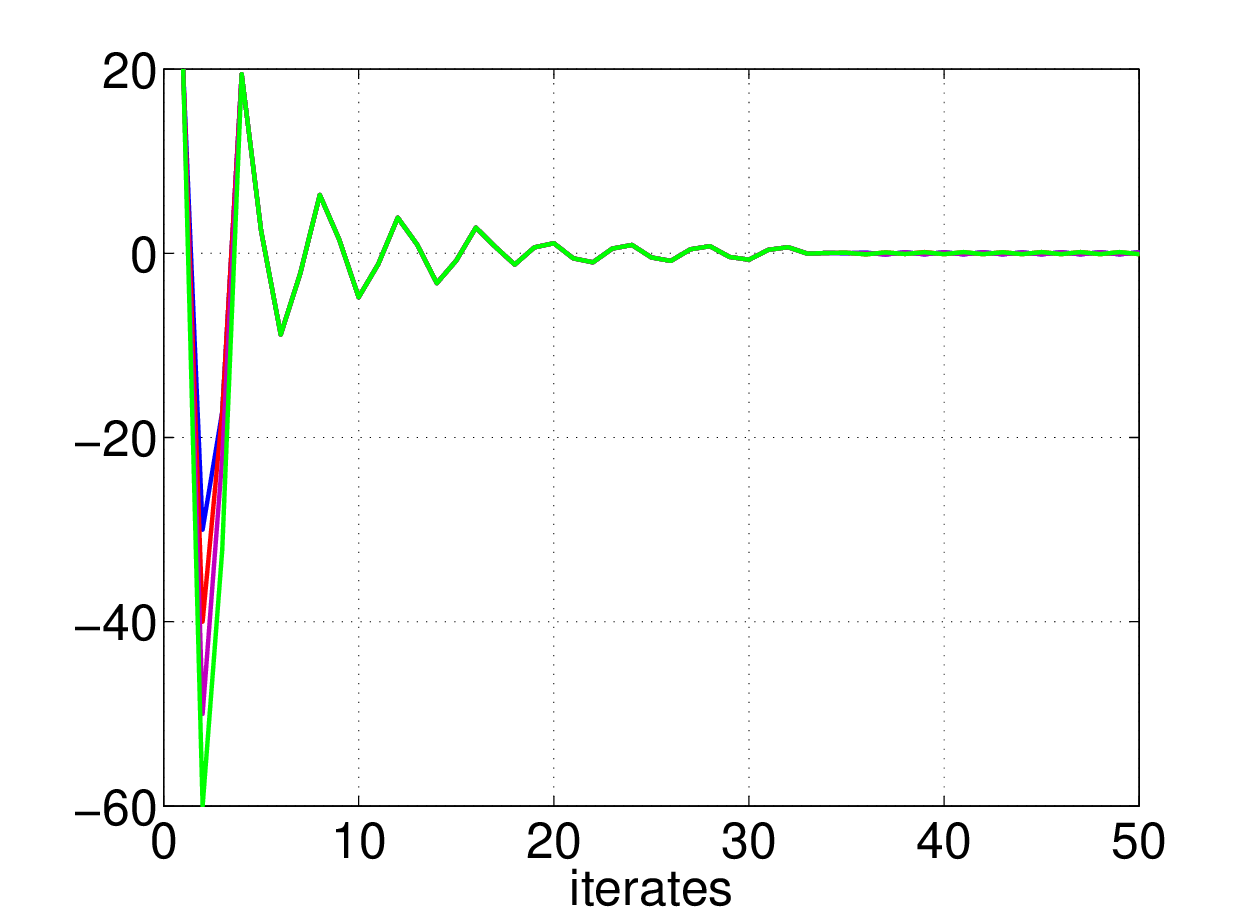}
  \caption{The evolution of $\zeta_{i,2}$} \label{fig_slop2_oscillation}
\end{figure}

\begin{figure}[h]
  \centering
  \includegraphics[width = .5\linewidth]{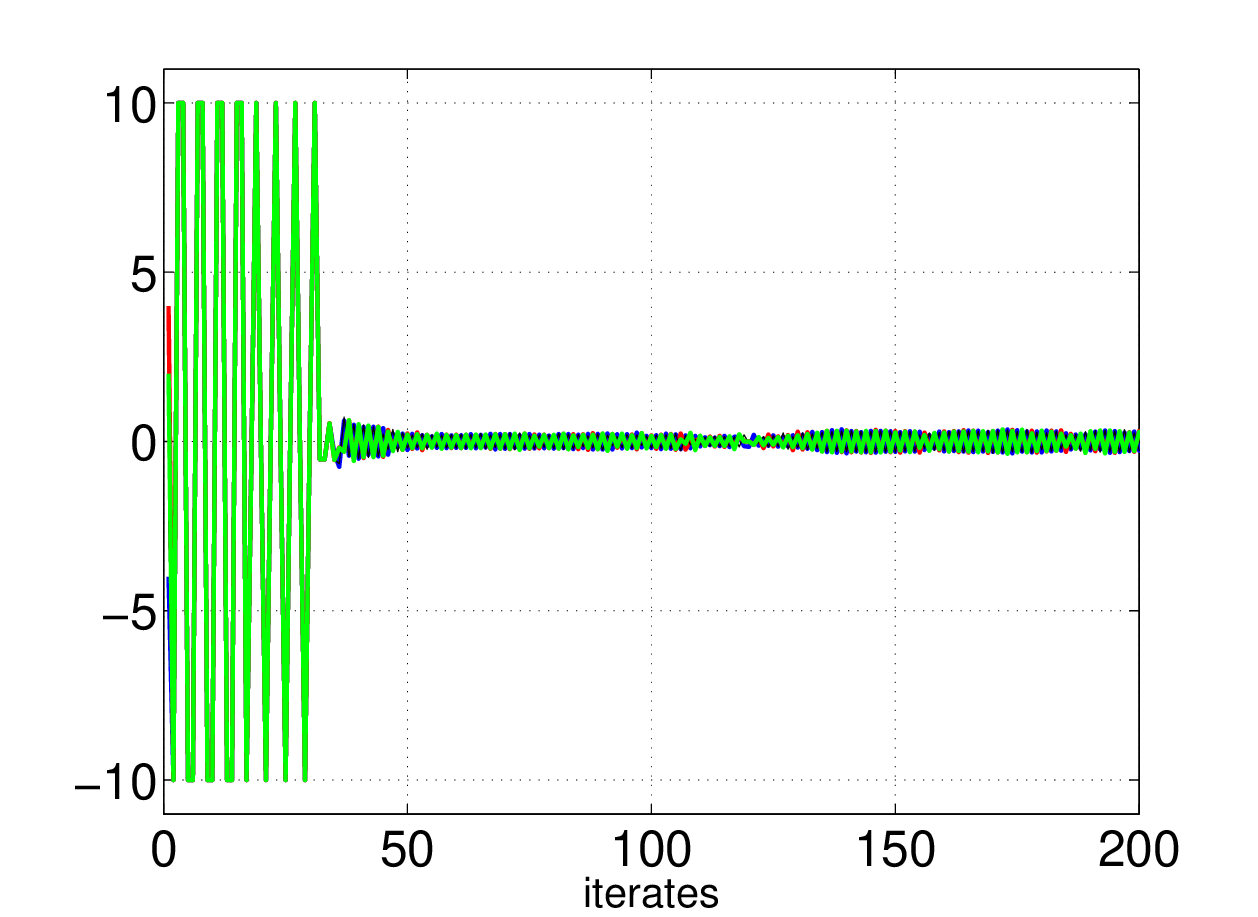}
  \caption{The primal estimates of $x^*(1)$} \label{fig_x1_oscillation}
  \includegraphics[width = .5\linewidth]{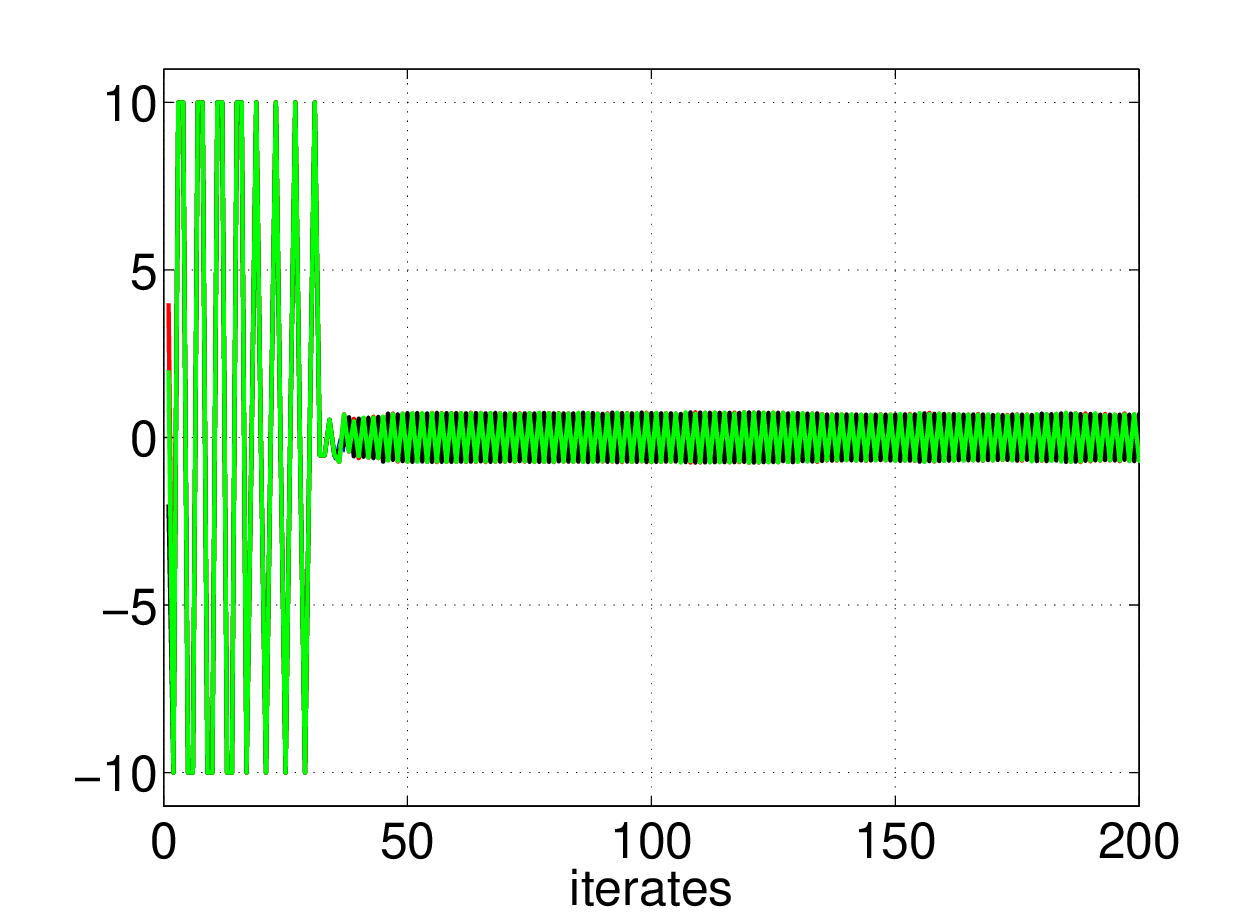}
  \caption{The primal estimates of $x^*(2)$} \label{fig_x2_oscillation}
\end{figure}



\begin{figure}[h]
  \centering
  \includegraphics[width = .75\linewidth]{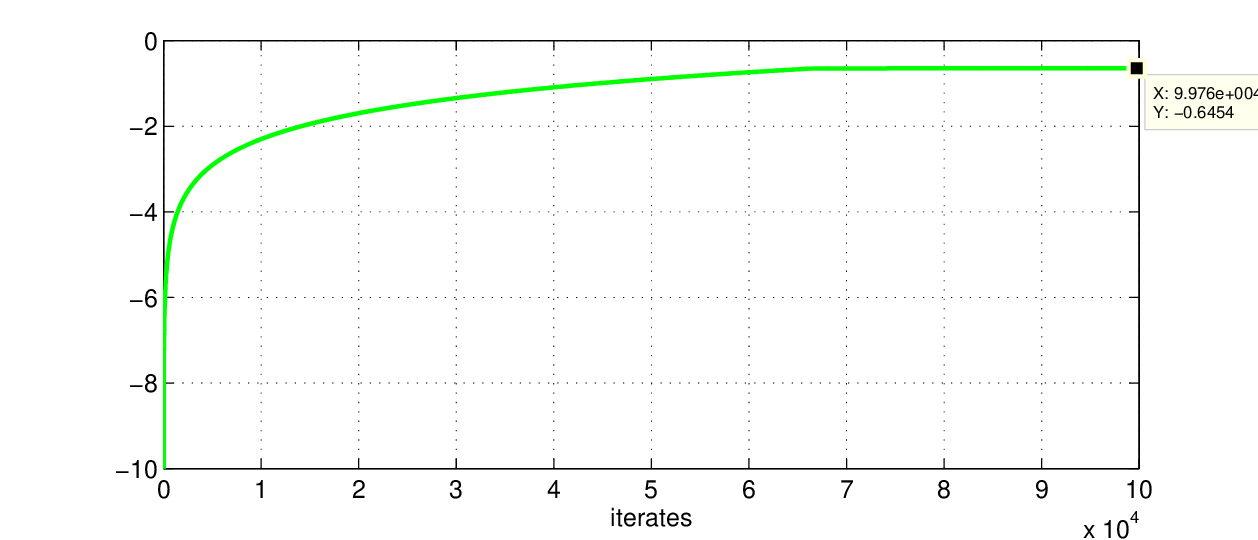}
  \caption{The primal estimates of $x^*(1)$ of the diffusion gradient method} \label{fig_grad_x1}
  \includegraphics[width = .75\linewidth]{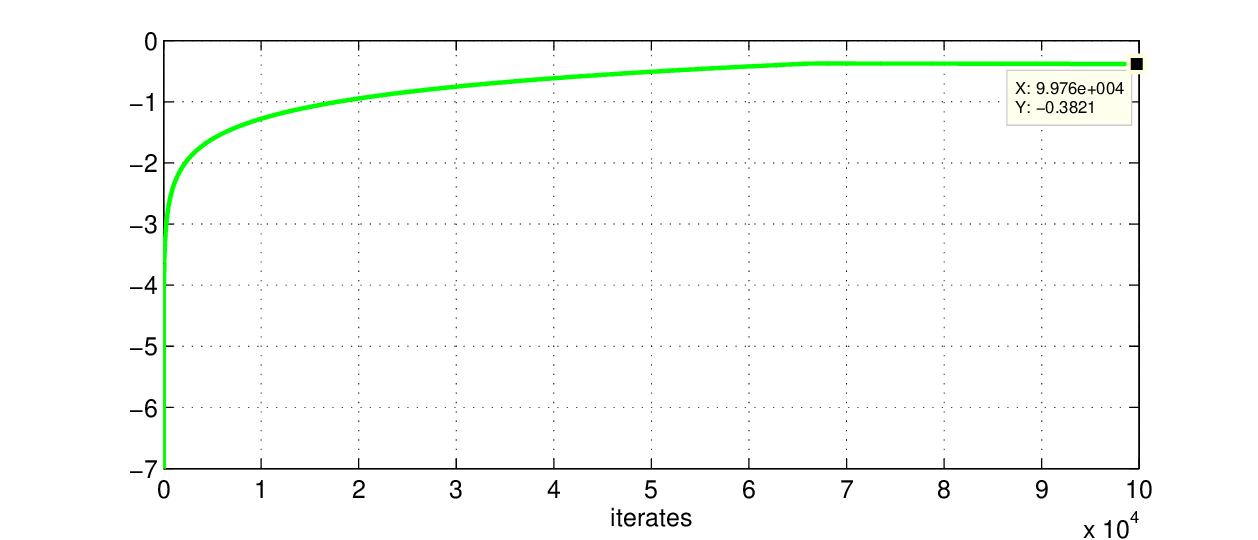}
  \caption{The primal estimates of $x^*(2)$ of the diffusion gradient method} \label{fig_grad_x2}
\end{figure}

\begin{figure}[h]
  \centering
  \includegraphics[width = .5\linewidth]{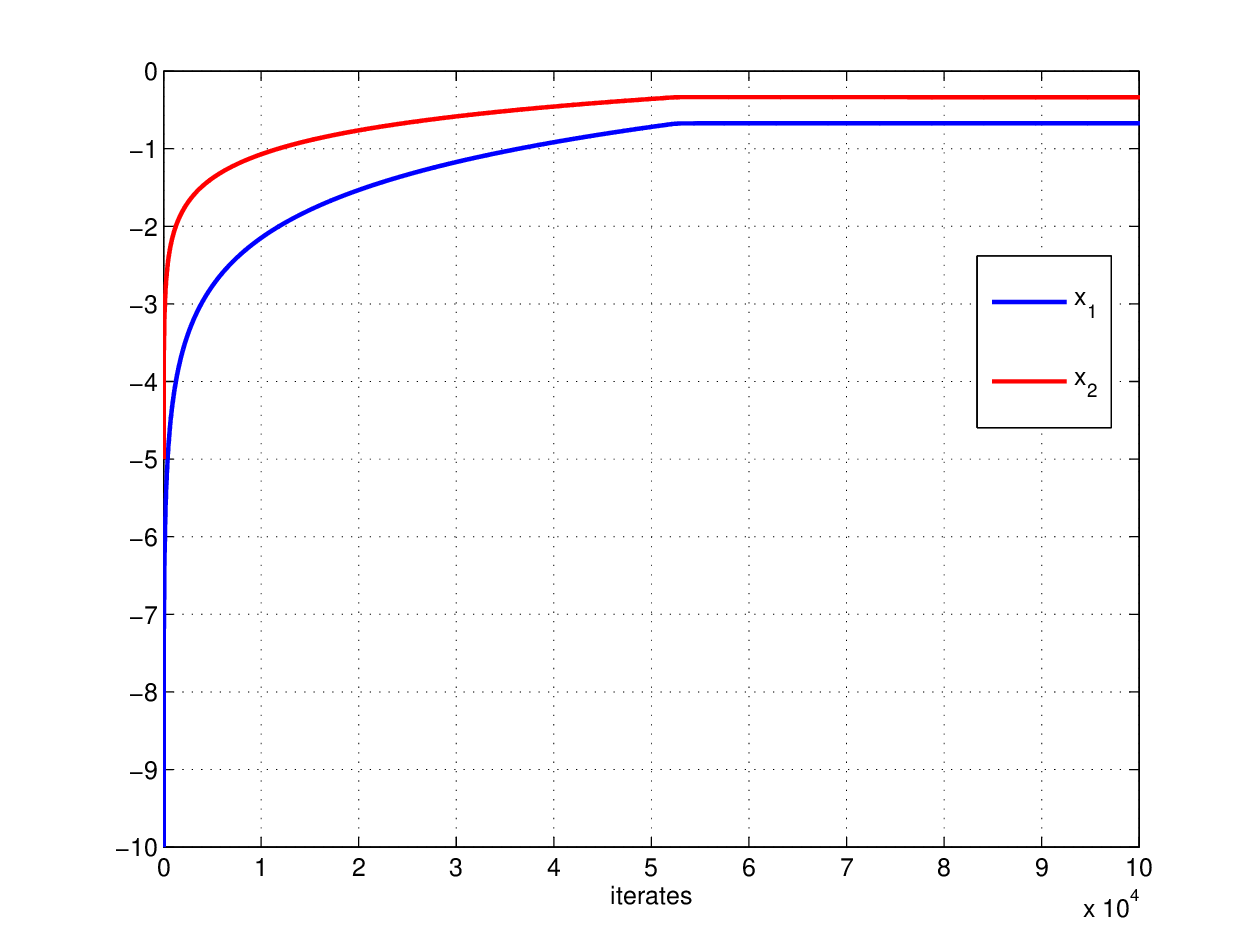}
  \caption{The primal estimates of $x^*(1)$ of the incremental gradient method} \label{fig_incremental}
\end{figure}



\begin{figure}[h]
  \centering
  \includegraphics[width = .75\linewidth]{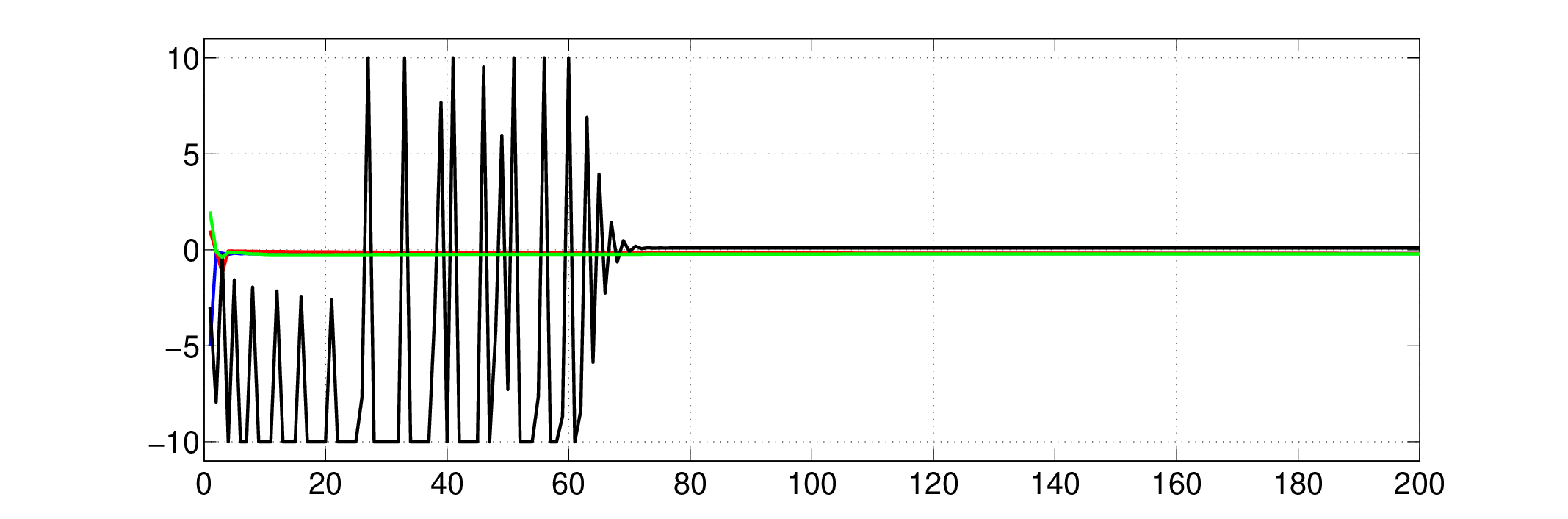}
  \caption{The primal estimates of $x^*(1)$ of quadratic programming} \label{fig_QP_x1}
  \includegraphics[width = .75\linewidth]{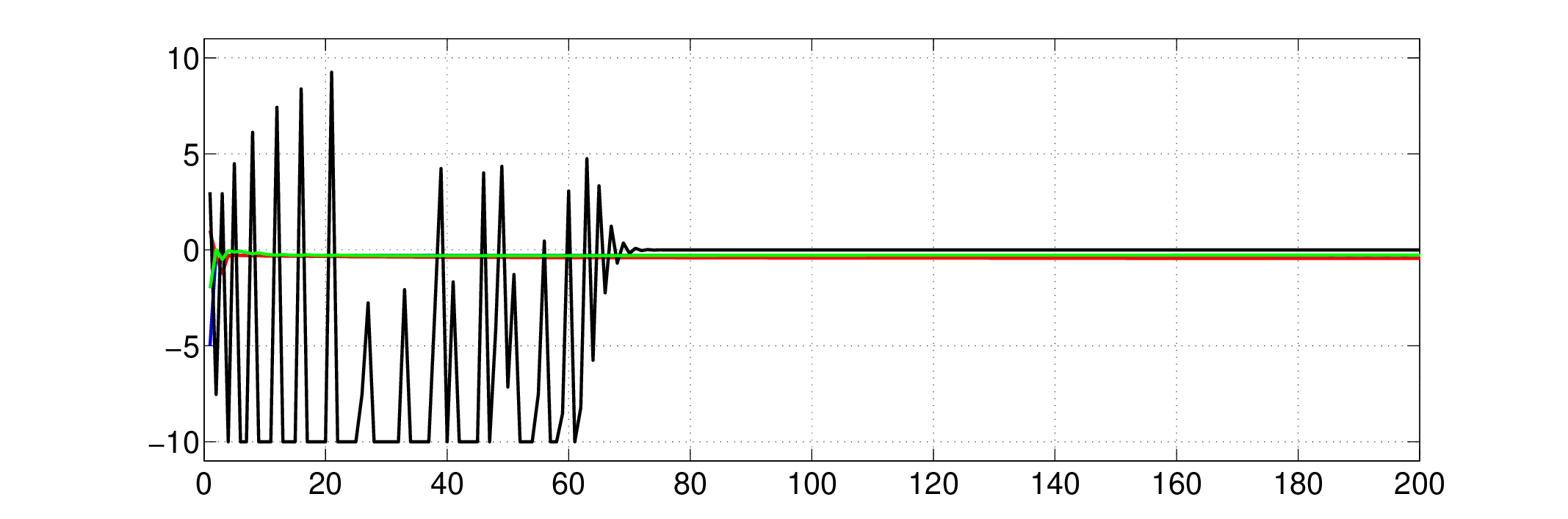}
  \caption{The primal estimates of $x^*(2)$ of quadratic programming} \label{fig_QP_x2}
  \includegraphics[width = .75\linewidth]{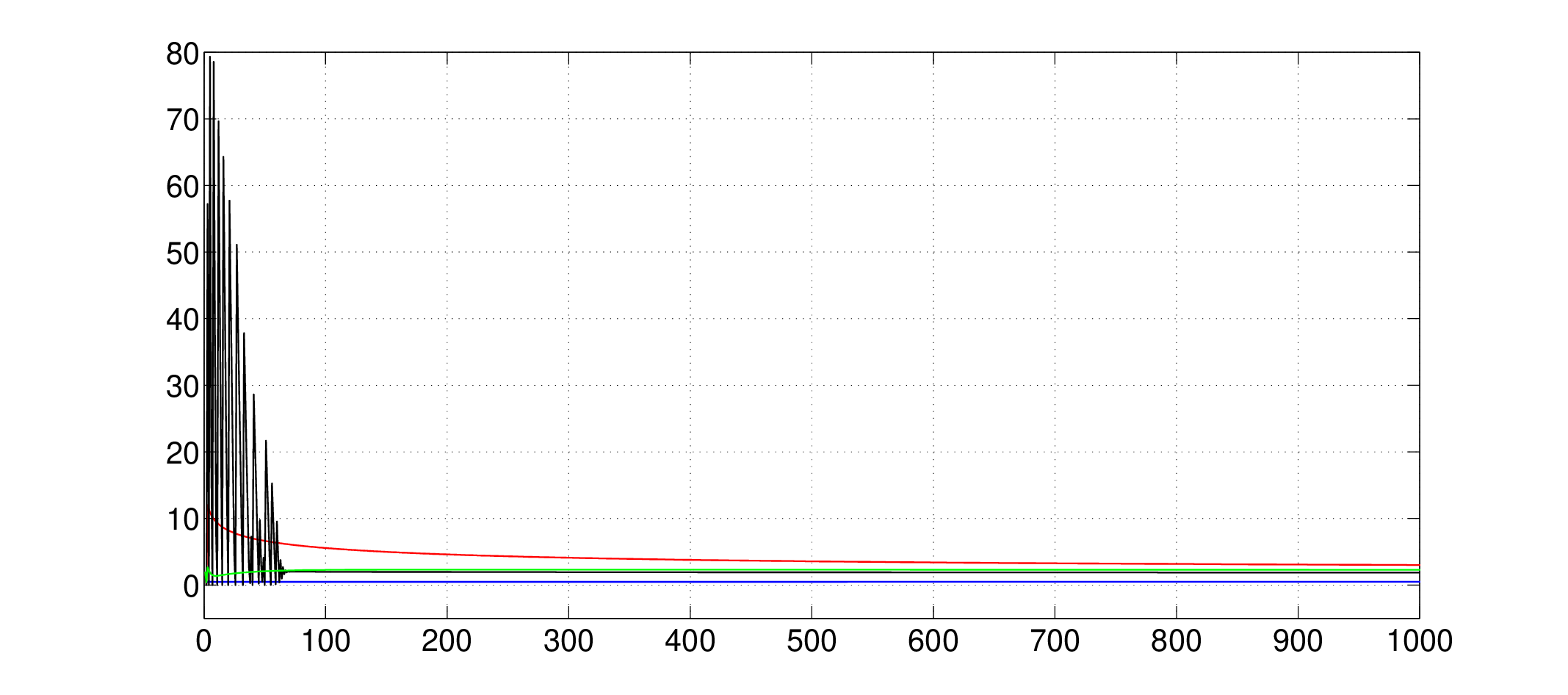}
  \caption{The dual estimates of $\mu^*_i$ of quadratic programming} \label{fig_QP_mu1}
\end{figure}

\end{document}